\documentclass[12pt,reqno,dvipsnames]{amsart}
 \usepackage[normalem]{ulem}
 \usepackage[english]{babel}
 \usepackage{latexsym} 
 \usepackage{mathtools,thmtools}
 \usepackage[all]{xy}
 \usepackage{amsfonts,amssymb,amsmath,amsthm,mathrsfs}
 \usepackage{pifont}  
 \usepackage{enumerate}
 \usepackage{dcolumn}
 \usepackage{comment}
\usepackage[inline]{enumitem}
\usepackage{etoolbox}
\usepackage{multicol} 
\usepackage[textsize=scriptsize]{todonotes}

\usepackage[bookmarks=false]{hyperref}  
\newcolumntype{2}{D{.}{}{2.0}}
\xyoption{2cell}
	
\makeatletter
\patchcmd{\@setaddresses}{\indent}{\noindent}{}{}
\patchcmd{\@setaddresses}{\indent}{\noindent}{}{}
\patchcmd{\@setaddresses}{\indent}{\noindent}{}{}
\patchcmd{\@setaddresses}{\indent}{\noindent}{}{}
\makeatother

 \usepackage{tikz}
\usetikzlibrary{arrows}

  \def\<{{\langle}} 
  \def\>{{\rangle}}

  \def\eps{\varepsilon}

  \def\note#1{{}}

  \def\note#1{}

   \def\cH{{\mathcal H}}   
  
  \def\cO{{\mathcal O}}
   
     \def\cT{{\mathcal T}}

  \def\lhom#1#2#3{{}{{\rm Hom}\sb{#1}(#2,#3)}}

  \def\lend#1#2{{{\rm End}\sb{#1}(#2)}} 
   
  \def\abs#1{\mathrm{Abs}(#1)}

  \def\End#1#2{{{\rm End}\sb{#1}(#2)}}

  \def\beq{\begin{equation}} 
  \def\eeq{\end{equation}}

  \def\id{\mathrm{id}} 
  \def\im{{\rm Im}}

  \def\ot{{\otimes}} 
   
  \def\Hom{\mbox{\rm Hom}\,}

 \def\bs{\backslash}

  \newcounter{zlist} 
  \newenvironment{zlist}{\begin{list}{(\arabic{zlist})}{ 
  \usecounter{zlist}\leftmargin2.5em\labelwidth2em\labelsep0.5em 
  \topsep0.6ex
  \parsep0.3ex plus0.2ex minus0.1ex}}{\end{list}}

  \newcounter{blist} 
  \newenvironment{blist}{\begin{list}{(\alph{blist})}{ 
  \usecounter{blist}\leftmargin2.5em\labelwidth2em\labelsep0.5em 
  \topsep0.6ex 
  \parsep0.3ex plus0.2ex minus0.1ex}}{\end{list}} 

  \newcounter{rlist}

\def\stac#1{\raise-.2cm\hbox{$\stackrel{\displaystyle\otimes}{\scriptscriptstyle{#1}}$}}

\def\cten#1{\raise-.2cm\hbox{$\stackrel{\displaystyle\reallywidehat{\otimes}}
{\scriptscriptstyle{#1}}$}}

  \newtheorem{proposition}{Proposition}[section]
  \newtheorem{lemma}[proposition]{Lemma} 
  \newtheorem{corollary}[proposition]{Corollary} 
  \newtheorem{theorem}[proposition]{Theorem} 
  
  \newtheorem{exlemma}[proposition]{Lemma and Example}
\theoremstyle{definition} 
  \declaretheorem[name=Definition,qed={$\lozenge$},sibling=proposition]{definition}
    
  \newtheorem{example}[proposition]{Example}

  \theoremstyle{remark} 
  \declaretheorem[name=Remark,qed={$\triangle$},sibling=proposition]{remark}

   \numberwithin{equation}{section}

\newcommand{\Ah}{\mathbf{Ah}}

\newcommand{\Ab}{\mathbf{Ab}}

\newcommand{\spin}{\mathbf{Spin}}

\newcommand{\quandle}{\mathbf{Qndl}}

\newcommand{\Mod}{\mbox{-}\mathbf{ Mod}}
\newcommand{\HMod}{\mbox{-} \mathbf{ HMod}}

\newcommand{\Aff}{\mathbf{Aff}}

\newcommand{\heap}{{\bf Hp}}

\newcommand{\Ann}{{\rm Ann}\,}

\def\ot{\otimes}

\def\FF{{\mathbb F}}

\def\NN{{\mathbb N}}

\def\ZZ{{\mathbb Z}}

\newcommand{\eE}{\mathrm{E}}

\newcommand{\gG}{\mathrm{G}}
\newcommand{\hH}{\mathrm{H}}

\newcommand{\rR}{\mathrm{R}}

\newcommand{\tT}{\mathrm{T}}

\newcommand{\Cc}{\mathcal{C}}

\newcommand{\Ff}{\mathcal{F}}

\newcommand{\Hh}{\mathcal{H}}

\newcommand{\Rr}{\mathcal{R}}

\def\Z{{\bf Z}}

\def\*C{{}^*\hspace*{-1pt}{\Cc}}

\def\text#1{{\rm {\rm #1}}}

\def\Set{\mathbf{Set}}

 \def\k{\mathbf{k}}

 \def\1{\mathbf{1}}

\def\id{\mathrm{id}}

\def\ahrd{\mathbf{Ah}}
\def\grp {\mathbf{Grp}}

\def\lto{\longmapsto}
\def\lra{\longrightarrow}
\def\lhom#1#2#3{\mathrm{Hom}_#1(#2,#3)}

    \def\eps{\varepsilon}

\def\1\mathbf{1}

\def\|#1{\overline{#1}}

\def\k{\Bbbk}
\def\stab{\mathrm{Stab}}
\def\act{\!\cdot\!}
\def\Sym#1{\mathfrak{S}_{#1}}

\newcommand{\MAbs}{T\mbox{-}\mathbf{Abs}}
\newcommand{\GMod}{\mbox{-}\mathbf{Grp}}

\allowdisplaybreaks

\begin{document}
\baselineskip=14.4pt
\title[Heaps of modules]{Heaps of modules and affine spaces}

\author[S. Breaz]{Simion Breaz}

\address{"Babe\c s-Bolyai" University, Faculty of Mathematics and Computer Science, Str. Mihail Kog\u alniceanu 1, 400084, Cluj-Napoca, Romania}
\email{bodo@math.ubbcluj.ro}

\author[T. Brzezi\'nski]{Tomasz Brzezi\'nski}

\address{
Department of Mathematics, Swansea University, 
Swansea University Bay Campus,
Fabian Way,
Swansea,
  Swansea SA1 8EN, U.K.\ \newline 
Faculty of Mathematics, University of Bia{\l}ystok, K.\ Cio{\l}kowskiego  1M,
15-245 Bia\-{\l}ys\-tok, Poland}

\email{T.Brzezinski@swansea.ac.uk}

\author[B. Rybo{\l}owicz]{Bernard  Rybo{\l}owicz}

\address{
Department of Mathematics, 
Heriot-Watt University, 
Edinburgh EH14 4AS, U.K.}

\urladdr{\url{https://sites.google.com/view/bernardrybolowicz/}}
\email{B.Rybolowicz@hw.ac.uk}

\author[P. Saracco]{Paolo Saracco}

\address{
D\'epartement de Math\'ematique, Universit\'e Libre de Bruxelles, 
Bd du Triomphe, B-1050 Brussels, Belgium.}

\urladdr{\url{https://sites.google.com/view/paolo-saracco}}
\urladdr{\url{https://paolo.saracco.web.ulb.be}}

\email{paolo.saracco@ulb.be}



\begin{abstract}
A notion of {\em heaps of modules} as an affine version of modules over a ring or, more generally, over a truss, is introduced and studied. Basic properties of heaps of modules are derived. Examples arising from geometry (connections, affine spaces) and algebraic topology (chain contractions) are presented. Relationships between heaps of modules and modules over a ring  and affine spaces are revealed and analysed. 
\end{abstract}    
\date\today
\maketitle
\tableofcontents

\section{Introduction}
A recent attempt \cite{BreBrz:Bae} to extend the Baer-Kaplansky theorem, relating isomorphisms of abelian $p$-groups to isomorphisms of their endomorphism rings \cite[Theorem 16.2.5]{Fu15}, to all abelian groups and to modules over rings led first to realise that the rings of group endomorphisms should be replaced by the trusses of the corresponding heap endomorphisms \cite{Brz:par}, and then that endomorphisms of modules should be replaced by endomorphisms of some more general module-like structures. The aim of the present text is to introduce and study in a systematic way such structures, which we term {\em heaps of modules}. In particular, we will enlarge upon their unexpected geometric interpretation as affine spaces and modules.

We begin in Section~\ref{sec.prelim} by giving an overview of heaps \cite{Bae:ein,Pru:the} and by carefully explaining why they can be understood as affine versions of groups. Next, we recall the definition and elementary properties of abelian heaps with an associative multiplication distributing over the ternary heap operation, which are called {\em trusses}, and their modules, that are abelian heaps on which the truss acts through a binary operation \cite{Brz:tru,Brz:par}. 
In particular, we recall how with every module $M$ over a truss $T$ and with every element $e\in M$, one can associate the {\em induced $T$-module} structure $(M,\triangleright_e)$ on $M$, which plays an important role in the paper. 
We also introduce the notions of {\em stabilizer} and {\em annihilator} for a module over a truss $T$ and we study how these are related with the corresponding constructions for the induced $T$-module structures. 
Finally, we describe abelian groups with a structure of module over a truss $T$, called {\em $T$-groups}, which will represent the linear core of affine spaces over trusses.

Section~\ref{sec:mod} is devoted to the definition of heaps of modules over a truss $T$, Definition~\ref{def:hom}, and their elementary properties. The key feature here is that a truss $T$ acts on its (abelian) heap of modules $M$ not through a binary operation $T\times M\lra M$ (as is the case for a $T$-module)  but by a ternary operation $\Lambda: T\times M\times M\lra M$ instead. 
In Subsection~\ref{subs:first} we study some first properties of heaps of modules and, in particular, we show the mutual independence of axioms. We introduce sub-heaps of modules and explain how the corresponding equivalence relation is a congruence of heaps of modules and, conversely, how equivalence classes of a congruence of heaps of modules are sub-heaps of modules. We show how fixing a middle term in the action $\Lambda$ of a heap of $T$-modules yields a $T$-module  $(M,\act_e) = \Lambda(-,e,-)$ (so we have a {\em heap} of modules indeed), exactly as fixing a middle entry in a heap produces a group. Next, we prove that there is a correspondence between heaps of modules and induced actions and how this provides a functor $\Hh$ from the category of $T$-modules to the category of heaps of $T$-modules. 
In Subsection~\ref{ssec:UZprop} we consider {\em stabilizers} and {\em annihilators} for heaps of modules over a truss $T$, which play a key role in correctly identifying affine modules over a ring among heaps of modules over a suitably related truss. They extend the corresponding notions introduced for modules in Section~\ref{sec.prelim}. 
Heaps of modules with non-empty stabilizers are said to be {\em isotropic}, while those with non-empty annihilator are said to be {\em contractible}. 
We end Section~\ref{sec:mod} with an explicit construction, extending \cite[Theorem~4.2]{BrzRyb:con}, that provides a cross-product truss structure on the product $M\times T$ of a heap of $T$-modules $M$ and $T$ itself for every element $e$ of $M$.

Section \ref{sec:aff} contains main results showing how heaps of modules over a truss are intimately related with affine geometry and how they provide an algebraic description of affine modules over a ring or affine spaces over a field.
We start Section \ref{sec:aff} by showing, in Proposition~\ref{prop:newTHmodTMod}, why homomorphisms of heaps of $T$-modules $f\in T\HMod(M,N)$ are, in fact, translations of homomorphism of $T$-groups from $(M,\act_m)$ to $(N,\act_{n})$, for an arbitrary choice of $m\in M$ and $n\in N$, exactly as affine maps are translations of linear morphisms. 
In Subsection~\ref{ssec:aff}, we fully describe the affine nature of heaps of modules. After realizing that affine modules over a ring, as defined in \cite[page 45]{Ostermann-Schmidt}, are isotropic and contractible heap of $\tT(R)$-modules (see Proposition~\ref{prop:affinemods}), we present the definition of a $T$-affine space (Definition~\ref{def:Taffinespace}) as a straightforward extension of the classical definition of affine space over a field by replacing the free and transitive action of a vector space by a free and transitive action of a $T$-group. The main result of this section is Theorem~\ref{thm:affT}, which states that the categories of $T$-affine spaces and of heaps of $T$-modules are equivalent. This result leads us to few immediate conclusions. 
For example, Corollary~\ref{cor:classic} asserts that the category of isotropic $\star$-affine spaces, where $\star$ denotes the singleton truss, is equivalent to the category of isotropic heaps of $\star$-modules, that is, abelian heaps. Thus, we deduce that the categories of inhabited abelian heaps and of torsors over abelian groups are equivalent. 
We conclude this section by proving that the category of affine spaces over a field $\mathbb{F}$ is equivalent to the full subcategory of heaps of $\tT(\mathbb{F})$-modules consisting of inhabited isotropic contractible heaps of $T$-modules (see Corollary~\ref{cor:afffield}). All of that together wells up in a sentence: heaps of $T$-modules are affine versions of $T$-groups, that is, heap of modules are the natural extension of affine spaces to modules over rings or trusses.

Section~\ref{sec:appl} contains various examples and some applications of heaps of modules. In particular, Subsection \ref{ssec:BK} is devoted to prove a Baer-Kaplansky theorem for $T$-groups by taking advantage of the results from Section \ref{sec:aff}. Subsection~\ref{ssec:spindle} studies the appearances of heaps of modules in algebraic systems related to the classification of knots, such as spindles and quandles. For instance, we show how affine spindle structures on an abelian group can be organised into a heap of modules (Example \ref{ex:HModaffspindle}). More generally, to any element $u$ of a truss $T$ one can assign a fully faithful functor from the category of heaps of $T$-modules to the category of spindles. If, in addition the element $u$ has a suitable companion, this functor has its image in the category of (entropic) quandles; see Theorem~\ref{thm:quandle}. Since to every spindle (quandle) one can associate a solution to the set-theoretic Yang-Baxter equation, we conclude that heaps of modules yield such solutions. In Subsection~\ref{ssec:split}, finally, we provide examples of heaps of modules that arise from (non-commutative) geometry and homological algebra. We show that non-commutative connections and hom-connections can be organised into heaps of modules, too, and -- as a consequence of Theorem~\ref{thm:quandle} -- how they give rise to spindles (or, if the element inducing the spindle is a unit, quandles) and hence to solutions to the set-theoretic Yang-Baxter equation. Finally, we construct heaps of modules consisting  of splittings and retractions of short exact sequences, and -- more generally -- consisting of chain contractions. 

We use the following categorical conventions and notation. The set of morphisms with domain $A$ and codomain $B$ in a category $\mathbf{C}$ is denoted by $\mathbf{C}(A,B)$. In case of the same domain and codomain $A$, the monoid of endomorphism of $A$ is denoted simply by $\mathbf{C}(A)$. We write $\mathbf{C}(A)^\times$ for the group of units in $\mathbf{C}(A)$, that is, automorphisms of $A$. In general, $M^\times$ denotes the group of units in a monoid $M$. The category of groups is denoted by $\grp$ and its full subcategory of abelian groups is denoted by $\Ab$.\label{p-grp}\label{p-cat-ab} We denote by $\End{}A=\Ab(A,A)$ the endomorphism ring of an abelian group $A$. In case of left (or right) modules $M$, $N$ over a ring $R$, the abelian groups of homomorphisms are denoted by $\Hom_ R(M,N)$, and the endomorphism ring of, say $M$, by $\End RM$. In anticipation of a possible confusion arising from the wealth of notation employed in this text, and for the convenience of the reader, we include the list of frequently used symbols in Appendix~\ref{appendix}.

\section{Preliminaries and first results}\label{sec.prelim}

In this section we recall the notions of heap, truss, module over a truss and their morphisms, which will be needed throughout the paper. We also introduce and discuss the notions of isotropy and contracting paragons for a module and of abelian group with action of a truss, which will have a role to play in relating affine modules over rings or trusses with heaps of modules (see \S\ref{sec:aff}).

\subsection{Heaps and their morphisms}\label{ssec:heaps}
We start with the definition of a heap, which can trace its roots back to \cite{Bae:ein} and \cite{Pru:the}.
\begin{definition}
A {\em heap} is a set $H$ with a ternary operation $[-,-,-]\colon H\times H\times H\lra H$ such that for all $a,b,c,d,e\in H$ the following axioms hold:
\begin{align}
[a,b,[c,d,e]] & = [[a,b,c],d,e] & \textrm{(Associativity)}, \label{eq:assoc} \\
[a,b,b] & = [b,b,a] = a & \textrm{(Mal'cev identities)}. \label{eq:malcev}
\end{align}
Moreover, if $[a,b,c]=[c,b,a]$ for all $a,b,c\in H$, then $H$ is called an {\em abelian heap}.
\end{definition}

For a heap $H$, it can be checked that also the following associativity holds
\begin{equation}\label{eq:middleass}
[a,b,[c,d,e]] = [a,[d,c,b],e],
\end{equation}
for all $a,b,c,d,e \in H$  (see \cite[Lemma 2.3(2)]{Brz:par}).

\begin{definition}
A {\em homomorphism of heaps} is a function $f\colon H\lra H'$ between heaps $H$ and $H'$ which preserves the ternary operation, that is, for all $a,b,c\in H$ we have $f([a,b,c])=[f(a),f(b),f(c)]$. We denote by $\heap$ the category of heaps and their homomorphisms and by $\Ah$ the full subcategory of $\heap$ consisting of abelian heaps.\label{p-cat-heaps}\label{p-cat-ah}
\end{definition}

Among all homomorphisms of heaps a special role is played by {\em translation automorphisms}, defined for all $a,b\in H$ by the formula
\begin{equation}\label{eq:trans}
\tau_a^b\colon H\lra H, \qquad x\lto [x,a,b].
\end{equation} \label{page-trans-auto}
The set of all translation automorphisms together with the identity of $H$  is denoted by $\mathrm{Trans}(H)$. Since the inverse of $\tau_a^b$ is given by $\tau_b^a$, the set $\mathrm{Trans}(H)$ is closed under inverses. Furthermore, one easily proves that, for all $a,a',b,b'\in H$,
\begin{equation}\label{comp.trans}
    \tau_a^{b}\circ \tau_{a'}^{b'} \stackrel{\eqref{eq:assoc}}{=} \tau_{a'}^{[b',a,b]} \stackrel{\eqref{eq:middleass}}{=} \tau_{[a,b',a']}^{b}.
\end{equation}
Therefore, $\mathrm{Trans}(H)$ \label{p-trans-group} is a subgroup of the automorphism group of $H$, which is called the \emph{translation group} of $H$. The translation group of $H$ is abelian if $H$ is abelian. In view of \eqref{eq:trans}, for any homomorphism of heaps $f:H\lra H'$, the map
\begin{equation}\label{trans.funct}
\mathrm{Trans}(f)\colon \mathrm{Trans}(H)\lra \mathrm{Trans}(H'), \qquad \tau_a^b\lto \tau_{f(a)}^{f(b)},
\end{equation}
is a homomorphism of groups. This gives a functor $\mathrm{Trans}\colon \heap\lra \grp$ that restricts to a functor $\Ah\lra\Ab$, that we denote by $\mathrm{Trans}$ again.

A \emph{sub-heap} of a heap $H$ is a subset $S$ closed under the ternary operation. Given a non-empty sub-heap $S$ of $H$ one can define the sub-heap equivalence relation by $x\sim_Sy$ provided $[x,y,s]\in S$, for all $s\in S$. \label{page-sub-heap-congr} The quotient set is denoted by $H/S$. In case of an abelian heap, the sub-heap relation is a congruence and consequently $H/S$ is an abelian heap too. Furthermore, each equivalence class is a sub-heap of $H$. Any congruence relation of abelian heaps is a sub-heap relation.

With every group $(G,\cdot,e)$ we can associate a heap $\hH(G) = (G,[-,-,-])$ where $[x,y,z] = xy^{-1}z$ for all $x,y,z\in G$. \label{functor-heap} This assignment is functorial, thus yielding a functor 
$\hH \colon  \grp \to \heap$. \label{functor-Hh}
In the opposite direction, with every non-empty heap $H$ and an element $e\in H$, we can associate a group $\gG(H;e)=(H,[-,e,-])$, \label{groups-retract} where  the binary operation is acquired by fixing the middle variable in the ternary operation ($e$ is the neutral element for this operation). The group $\gG(H;e)$ is called the {\em  $e$-retract of the heap $H$}. Note that for all heaps $H$ and $e\in H$, $\hH(\gG(H;e))=H$, while for every group $G$ and $e\in G$, $\gG(\hH(G);e)) \cong G$ with the equality if $e$ is the neutral element of $G$. 

The assignment of a retract to a heap and an element is not functorial. The subsequent results explore the relationship between morphisms of heaps and morphisms of the associated retracts. In particular, they clarify why heaps can be understood as affine versions of groups.

\begin{proposition}\label{prop:AhAb}
Let $H,H'$ be non-empty heaps, let $e \in H$ and $e'\in H'$, and let $f \colon  H \lra H'$ be a function. Then $f \in \grp\big(\gG(H;e),\gG(H';e')\big)$ if and only if $f$ is in $\heap(H,H')$ and $f(e) = e'$.
\end{proposition}

\begin{proof}
By definition $f \in \grp\big(\gG(H;e),\gG(H';e')\big)$ if and only if
\begin{equation}\label{eq:AhAb1}
f\big([x,e,y]\big) = \big[f(x),e',f(y)\big], \qquad \textrm{for all }x,y \in H.
\end{equation}
Clearly, if $f \in \heap(H,H')$ and $f(e) = e'$ then \eqref{eq:AhAb1} is satisfied. 

Conversely, if $f \in \grp\big(\gG(H;e),\gG(H';e')\big)$ then
\[f = \hH(f) \in \heap\big(\hH(\gG(H;e)),\hH(\gG(H';e'))\big) = \heap(H,H')\]
and $f(e) = e'$.
\end{proof}

\begin{corollary}[{\cite[Lemma 2.1]{BreBrz:Bae}}]\label{cor:AhAb}
For non-empty heaps $H,H'$, a function $f \colon  H \lra H'$ is a morphism of heaps  if and only if $\tau_{f(e)}^{e'} \circ f \in \grp\big(\gG(H;e),\gG(H';e')\big)$, for all $e \in H$ and $e'\in H'$. In particular, $f \in \heap(H,H')$ if and only if $f \in \grp\big(\gG(H;e),\gG(H';f(e))\big)$, for all $e \in H$.
\end{corollary}

\begin{proof}
Recall that $\tau_{f(e)}^{e'} \in \heap(H')^\times$ and observe that $\big(\tau_{f(e)}^{e'} \circ f\big)(e) = e'$. Therefore, in view of Proposition \ref{prop:AhAb}, $\tau_{f(e)}^{e'} \circ f \in \grp\big(\gG(H;e),\gG(H';e')\big)$ if and only if $\tau_{f(e)}^{e'} \circ f \in \heap\big(H,H'\big)$, if and only if $f \in \heap\big(H,H'\big)$.
\end{proof}

\begin{remark}\label{rem-heap-morphism}
Let $f\colon H\lra H'$ be a morphism of heaps. It follows from Corollary \ref{cor:AhAb} that for any group operation associated to $H$, there exists a group operation associated to $H'$ such that $f$ is a morphism of groups with respect to these operations.
\end{remark}

\begin{corollary}\label{cor:mapsheaps}
Let $(G,\cdot,1_G)$ and $(H,\cdot,1_H)$ be groups. Then
\[
f\in \heap(\hH(G),\hH(H)) ~\Leftrightarrow~ \tau_{f(1_G)}^{1_H}\circ f=[x\lto f(x)f(1_G)^{-1}]\in \grp(G,H).
\]
\end{corollary}

\begin{proof}
Since $G = \gG\big(\hH(G);1_G\big)$ and $H = \gG\big(\hH(H);1_H\big)$, the statement follows from Corollary \ref{cor:AhAb}.
\end{proof}

\begin{proposition}\label{prop:affine}
The category $\grp$ is isomorphic to the under category $\star/\heap$.
\end{proposition}

\begin{proof}
By Corollary \ref{cor:mapsheaps}, the functor $\hH\colon \grp \lra \heap$ induces a fully faithful functor $\hH_\star\colon\grp \lra \star/\heap$ sending every group $G$ to the object $(\hH(G),1_G\colon\star \lra \hH(G))$. In view of Proposition \ref{prop:AhAb}, also the assignment $(H, e\colon \star \lra H) \lto G(H;e)$ induces a fully faithful functor, which is inverse to $\hH_\star$.
\end{proof}

\begin{remark}\label{rem:abelianization}
The properties from Proposition \ref{prop:AhAb}, Corollary \ref{cor:AhAb}, Corollary \ref{cor:mapsheaps} and Proposition \ref{prop:affine} hold unchanged in the abelian case, that is, if we substitute $\grp$ with $\Ab$ and $\heap$ with $\Ah$.
\end{remark}

{
\begin{example}\label{ex:heaps}
Let us present a few elementary examples of heaps.
\begin{enumerate}[label=(\arabic*),ref=(\arabic*),leftmargin=0.8cm]
 \item\label{item:exheaps0} An empty set $\varnothing$ is a heap with a trivial ternary operation $\varnothing \times \varnothing\times \varnothing=\varnothing\lra \varnothing$.
    \item\label{item:exheaps1} 
    Any singleton set $\star \coloneqq \{*\}$ together with $[*,*,*] = *$ is a heap and it is the terminal object in $\heap$. \label{page-star}
    \item The set $H \coloneqq 2\ZZ+1$ of odd numbers is an abelian heap with 
    \[[2n+1,2m+1,2p+1] \coloneqq 2(n-m+p)+1\]
    for all $n,m,p \in \ZZ$. Its retract at $1$ is a group with respect to 
    \[(2m+1)\cdot(2n+1) = 2(m+n)+1\]
    for all $m,n\in \Z$ and with neutral element $1$. Obviously, the latter is isomorphic to $\ZZ$ via the map $2m+1\lto m$.
    \item\label{item:exheaps3} More generally, a subset $S$ of a group $G$ is a non-empty sub-heap of $\hH(G)$ if and only if it is a coset for some subgroup $G'$ of $G$ (see \cite[Theorem 1]{Certaine}) and any heap $H$ can be realised as a coset of a certain group $G$. 
    
    In fact, one may always consider the group $G \coloneqq \mathrm{Gr}_*(H) = \gG(H\boxplus \star;*)$ from \cite[\S3]{Ryb-Free} obtained by adding a neutral element $*$ to the non-empty heap $H$ via the coproduct of heaps $\boxplus$. Then the canonical injection $\iota_H\colon  H \to G$ into the coproduct allows one
    to realise $H$ as a subset of $G$. For any $x \in H$, if we consider the restriction of 
    \[\tau_x^*\colon G \lra G,\qquad g \lto [g,x,*] = gx^{-1},\]
    to $H$ and we set $G'\coloneqq \tau_x^*(H) = Hx^{-1}$, then $G'$ is a subgroup of $G$. Indeed,  
    \[* = xx^{-1} \in Hx^{-1}, \quad yx^{-1}\cdot zx^{-1} = [y,x,z]x^{-1} \in Hx^{-1}\]
    and 
    \[\left(yx^{-1}\right)^{-1} = [x,y,x]x^{-1} \in Hx^{-1},\]
    for all $y,z \in H$. In this setting, it is clear that $H$ as a subset of $G$ coincides with the coset $G'x$.
    
    It is noteworthy that for all $y,z \in H$,
    \[\tau_*^x\left(\tau_x^*(y)\cdot \tau_x^*(z)\right) = \left[\left[y,x,*\right],*,\left[z,x,*\right],*,x\right] = [y,x,z] = y\cdot z \in \gG(H;x).\]
    Hence the group structure on $H$ obtained by transport of the group structure on $G'$ along $\tau_x^*$ is exactly the one of the retract of $H$ at $x$.
\end{enumerate}
\end{example}
}

\subsection{Trusses and modules}

Here we recall briefly some notions from \cite{Brz:par}.

\begin{definition}[Trusses and paragons]
A {\em truss} is an abelian heap $T$ together with an associative binary operation $\cdot\colon  T\times T\lra T$, $(a,b)\lto ab$ such that, for all $a,b,c,d\in T$,
\begin{enumerate}[leftmargin=0.8cm]
\item $a[b,c,d]=[ab,ac,ad]$ (left distributivity),
\item $[b,c,d]a=[ba,ca,da]$ (right distributivity).
\end{enumerate}
A truss $T$ is called \emph{unital} if there exists a neutral element for $\cdot$, usually denoted by $1_T$ or simply $1$. It is called \emph{commutative} if the binary operation $\cdot$ is commutative. A \emph{sub-truss} of $T$ is a sub-heap closed under multiplication.

A {\em homomorphism of trusses} is a heap homomorphism $f\colon T\lra T'$, that preserves binary operations, that is, $f(ab)=f(a)f(b)$ for all $a,b\in T$.

Given a truss $T$, a \emph{paragon} in $T$ is a non-empty sub-heap $P \subseteq T$ such that for all $p,q \in P$ and all $t \in T$, 
\[
[tp,tq,q] \in P \qquad \text{and} \qquad [pt,qt,q] \in P.
\]
Paragons are exactly equivalence classes of congruences in trusses, which in turn always arise as sub-heap relations by paragons.

Given a truss $T$, a \emph{two-sided ideal} in $T$ is a non-empty sub-heap $I \subseteq T$ such that for all $x \in I$ and all $t \in T$, $tx \in I$ and $xt \in I$.
\end{definition}

\begin{remark}
Observe that in the definitions of paragon and ideal, the word non-empty appears. Since the empty relation is not an equivalence relation, we cannot connect empty sub-heaps with congruences on heaps and trusses. Nevertheless, we still consider the empty set as a sub-heap and sub-truss.
\end{remark}

\begin{remark}\label{rem:para}
For any morphism of trusses $f:T\lra T'$ and $e\in \im{f}$, the inverse image $f^{-1}(e)$ is a paragon in $T$. In fact every paragon in $T$ arises in this way.
\end{remark}

\begin{example}\label{ex:truss}
All the examples from Example \ref{ex:heaps} can be adapted to provide examples of trusses. For instance, for any (unital) ring $R$, the abelian heap $\hH(R)$ is a (unital) truss with respect to the same product of $R$. We denote it by $\tT(R)$.\label{p-ass-truss}

Let us focus, in particular, on example \ref{item:exheaps3}. By summarising one of the main messages of \cite[Part 2]{AndBrzRyb:ext}, given a (unital) truss $T$ one can always endow the abelian group $\gG(T\boxplus \star;*)$ with a unique (unital) ring structure $\cdot$ by declaring 
\[s\cdot t = st, \qquad *\cdot*=* \qquad \text{and} \qquad t\cdot * = * = * \cdot t,\] for all $s,t \in T$. Denote the resulting (unital) ring by $R$. Then $\iota_T \colon  T \to R$ allows us to identify $T$ with a subset of $R$. Moreover, if $T$ is non-empty then for any $e \in T$, $I \coloneqq \tau_e^*(T)$ is a two-sided ideal of $R$. Indeed,  we already know it is a subgroup and moreover one may check that the elements $[r \cdot t,r\cdot e,e]$ and $[t\cdot r,e\cdot r,e]$ of $R$ are actually in $T$ for all $r\in R$, $e,t \in T$, whence
\[
\begin{gathered}
r\cdot[t,e,*] = [r\cdot t,r\cdot e,*] = \big[[r\cdot t,r\cdot e,e],e,*\big] = \tau_e^*\big([r\cdot t,r\cdot e,e]\big) \qquad \text{and} \\
[t,e,*]\cdot r = [t\cdot r,e\cdot r,*] = \big[[t\cdot r,e\cdot r,e],e,*\big] = \tau_e^*\big([t\cdot r,e\cdot r,e]\big).
\end{gathered}
\]

It is clear then that $T = \tau_*^e(I) = I + e \subseteq R$. Therefore, a subset $T$ of a ring $R$ is an equivalence class for some congruence on $R$ if and only if it is a paragon in $\tT(R)$ (whence, in particular, a sub-truss) and any non-empty truss can be realised as a residue class modulo $I$ for an ideal $I$ in a certain ring $R$.
\end{example}

\begin{example}\label{ex:endotruss}
For all abelian heaps $H,H'$, $\Ah(H,H')$ is an abelian heap with the pointwise operation. Furthermore the composition of morphisms distributes over this ternary operation. Consequently, $\Ah(H)$ is a unital truss, called the {\em endomorphism truss} and denoted by $\eE(H)$.\label{page-end-tr}
\end{example}

\begin{definition}
A {\em left module over a truss } $T$ or a {\em left $T$-module} is an abelian heap $M$ together with an action $\cdot \colon T\times M\lra M$ such that for all $t,t',t''\in T$ and $m,n,e\in M$,
\begin{enumerate}[leftmargin=0.8cm]
\item $t\cdot (t'\cdot m)= (tt')\cdot m$,
\item $[t,t',t'']\cdot m=[t\cdot m,t'\cdot m,t''\cdot m]$,
\item $t\cdot[m,n,e]=[t\cdot m,t\cdot n,t\cdot e]$.
\end{enumerate}
A left module $M$ over a unital truss $T$ is said to be \emph{unital} if $1_T \cdot m = m$ for all $m \in M$.
Analogously, one can define (\emph{unital}) {\em right $T$-modules}.
A \emph{morphism of $T$-modules} $f\colon  M \lra N$ is a heap homomorphism which is also $T$-linear, in the sense that $f(t \cdot m) = t \cdot f(m)$ for all $t \in T$ and $m \in M$. 

Modules over a truss $T$ and their morphisms form the category $T\Mod$.\label{p-t-mod}
\end{definition}

\begin{remark}
Given a truss $T$, a $T$-module can be equivalently described as an abelian heap $M$ together with a truss homomorphism $\phi\colon T \lra \eE(M)$. As in the case of modules over rings, the correspondence is given by $\phi(t)(m) = t\cdot m$.
\end{remark}

\begin{example}\label{ex:modules}
Here are some elementary examples of modules.
\begin{enumerate}[leftmargin=0.8cm]
\item The empty heap $\varnothing$ is a $T$-module with the unique operation $T\times \varnothing=\varnothing\to \varnothing$.
\item A singleton set $\star \coloneqq \{*\}$ together with the heap operation from Example \ref{ex:heaps}\ref{item:exheaps1} and the action of $T$ given by $t\cdot * = *$ for all $t\in T$, is the terminal object in $T\Mod$.
\item Given a ring $R$ and the associated truss $\tT(R)$, any $R$-module $M$ can be seen as a $\tT(R)$-module with respect to the $\hH(M)$ heap structure and the same $R$-action. We will denote the $\tT(R)$-module $\hH(M)$ by $\tT(M)$.\label{ex:assomodule}
    \item A truss $T$ is a left module over itself by multiplication. We refer to this as the \emph{regular action}.
    \item For any abelian heaps $H,H'$, $\Ah(H',H)$ is a left module over the endomorphism truss $\eE(H)$ and a right module over $\eE(H')$. The actions are by composition. Since $\Ah(\star,H)\cong H$ as heaps, $H$ is a left $\eE(H)$-module by evaluation.
\end{enumerate}
\end{example}

\begin{definition}\label{def:induced} 
Let $M$ be a left $T$-module. For $e \in M$, the action $\triangleright_e\colon  T \times M \lra M$ given by
\[t\ \triangleright_e\  m \coloneqq [t\cdot m,t\cdot e,e] \qquad \textrm{for all }m\in M, t \in T\]
is called the \emph{$e$-induced action} or the \emph{$e$-induced module structure} on $M$. We say that a subset $N\subseteq M$ is {\em an induced submodule} of $M$ if $N$ is a non-empty sub-heap of $M$ and $t\ \triangleright_e\  n \in N$ for all $t\in T$ and $n,e\in N$.
\end{definition}

Induced modules play a role similar to paragons in trusses. Every congruence class of a $T$-module $M$ is an induced submodule of $M$. A sub-heap relation corresponding to an induced submodule of $M$ is a congruence and every congruence arises in that way. Furthermore, for any epimorphism (that is, surjective morphism, \cite[Proposition 2.6]{BrzRyb:fun}) of $T$-modules $\pi:M\lra N$, $\pi^{-1}(n)\subseteq M$ is an induced submodule of $M$, for all $n\in N$, and $N\cong M/\pi^{-1}(n)$.

Recall also from \cite[Lemma 4.29]{Brz:par} that the induction procedure of Definition \ref{def:induced} stabilises after the first step. That is, if
\begin{equation}\label{eq:indchng}
t~{\triangleright_{f}}_{e} ~ m \coloneqq [t \triangleright_{f} m, t \triangleright_{f} e, e],
\end{equation}
then $t~{\triangleright_{f}}_{e} ~ m = t \triangleright_{e} m$ for all $t \in T$, $e,f,m \in M$.

\subsection{Isotropic and contractible modules}\label{ssec:isocontr}

In this subsection we introduce the notions of isotropy and contracting paragons for a module over a truss. They are preliminary notions to the introduction of isotropy and contracting paragons for heaps of modules \S\ref{ssec:UZprop} and their connection with affine modules over rings and trusses \S\ref{sec:aff}.

\begin{lemma}\label{lem:stab}
Let $T$ be a truss. For any left $T$-module $M$, the set
\[
\mathrm{Stab}(M) \coloneqq \{u\in T\; |\; u\cdot m = m, \; \mbox{for all $m\in M$}\}
\] 
is a sub-truss and, if non-empty, a paragon of $T$, called the \emph{stabilizer} or \emph{isotropy paragon}  of $M$.
\end{lemma}

\begin{proof}
The set $\mathrm{Stab}(M)$ is the inverse image of the identity endomorphism under the truss homomorphism $\phi:T\lra \eE(M)$, $\phi(t)(m)=t\cdot m$, and hence, if non-empty, a paragon in $T$ by Remark~\ref{rem:para}. Clearly, $\mathrm{Stab}(M)$ is closed under the multiplication (this can also be seen from the fact that it is an inverse image of an idempotent element of $\eE(M)$).
\end{proof}

\begin{definition}\label{def:isotropy}
A $T$-module $M$ is said to be \emph{isotropic} if $\stab(M)$ is non-empty. Isotropic $T$-modules form the full subcategory $T\Mod_{\mathbf{is}}$ of $T\Mod$.\label{page-iso-t-mod}
\end{definition}

\begin{proposition}\label{prop:unisim}
Let $M$ be a $T$-module. Then:
\begin{enumerate}[leftmargin=0.8cm]
\item\label{prop:unisim:eq1} For all $e\in M$, $ \stab(M)\subseteq \stab(M,\triangleright_e)$. 
\item\label{prop:unisim:eq2bis} If $f\colon M \lra N$ is a morphism of $T$-modules, then $\stab(M) \subseteq \stab(f(M))$.
\item\label{prop:unisim:eq2}
For all $e,f\in M$, $\stab(M,\triangleright_e) = \stab(M,\triangleright_f)$.
\item\label{prop:unisim:eq3} If $u\in\stab(M,\triangleright_e)$, then $[u,ut,t]\in \stab(M)$ for all $t \in T$. Hence, if $\stab(M,\triangleright_e)\neq \varnothing$, then  $\stab(M)\neq \varnothing$.
\item\label{prop:unisim:eq4} If $T$ is a truss with identity, then $M$ is a unital module if and only if $\stab(M,\triangleright_e)$ is a unital sub-truss of $T$.
\end{enumerate}
\end{proposition}

\begin{proof}
\eqref{prop:unisim:eq1} Let $u\in \stab(M)$. Then
$$
u\triangleright_e m=[u\act m,u\act e,e]=[m,e,e]=m
$$
for all $m\in M$. Thus, $u\in \stab(M,\triangleright_e)$.

\eqref{prop:unisim:eq2bis} For every $u \in \stab(M)$ and every $m \in M$ we have $u \cdot f(m) = f(u \cdot m) = f(m)$.

\eqref{prop:unisim:eq2} 
It follows from \eqref{prop:unisim:eq2bis} after recalling that $\tau_e^f\colon (M,\triangleright_e) \lra (M,\triangleright_f)$ is an isomorphism of $T$-modules (see \cite[Proposition 4.28]{Brz:par}).

\eqref{prop:unisim:eq3}  
Consider the heap homomorphism $E\colon T \lra M$, $t\lto t\cdot e$. 
For all $u\in \stab(M,\triangleright_e)$ and $m\in M$ we have $m = u\triangleright_em = [u\cdot m,u\cdot e,e]$ and hence
\begin{equation}\label{eq:sim-eq01bis}
u\cdot m =[m,e,E(u)] = [E(u),e,m].
\end{equation}
Therefore, 
$$[u,ut,t]\act m = \big[u\act m,ut\act m,t\act m\big] \stackrel{\eqref{eq:sim-eq01bis}}{=} \big[m,e,E(u),E(u),e,t\act m,t\act m\big] = m.
$$
Thus, $[u,ut,t]\in \stab(M)$ as required.

\eqref{prop:unisim:eq4} If $T$ has a unit $1_T$ and $M$ is unital, then $1_T\in \stab(M)$ and by statement \eqref{prop:unisim:eq1}  $1_T\in \stab(M,\triangleright_e)$. In the opposite direction if $1_T\in \stab(M,\triangleright_e)$, then by statement \eqref{prop:unisim:eq3} we have that $1_T = [1_T, 1_T^2,1_T]\in \stab(M)$ and hence  $M$ is a unital module. 
\end{proof}

\begin{remark}
Consider the restriction
\begin{equation}\label{eq:mapE}
\stab(M,\triangleright_e) \lra M, \qquad u\lto u\act e.
\end{equation}
of the heap morphism $E\colon T \lra M$ from the proof of Proposition \ref{prop:unisim}\eqref{prop:unisim:eq3}. By \eqref{eq:sim-eq01bis}, for all $u_1,u_2\in \stab(M,\triangleright_e)$ we have
\[
E(u_1u_2)=(u_1u_2)\act e=u_1\act E(u_2)=[E(u_1),e,E(u_2)]
\]
and hence the map \eqref{eq:mapE} is a truss homomorphism between $\stab(M,\triangleright_e)$ and the trivial brace $\mathrm{G}(M;e)$ (any abelian group $(G,+,0)$ is a brace with $m\cdot n \coloneqq m + n$).
\end{remark}

\begin{example}
The reverse inclusion in Proposition \ref{prop:unisim}\eqref{prop:unisim:eq1} does not necessarily hold. For example, take $T = \ZZ$ with multiplication $m \cdot n \coloneqq m + n$ and $M = T$ itself with the regular action. Then
$\stab(M) = \{0\}$ and $\stab(M,\triangleright_e) = \ZZ$, for all $e\in \ZZ$.
\end{example}

Recall that an \emph{absorber} for a $T$-module structure $(M,\cdot)$ on an abelian heap $M$ (or, simply, an absorber) is an element $e \in M$ such that $t \cdot e = e$ for all $t\in T$. By an absorber in $T$ we mean a (necessarily unique) two-sided absorber $0_T \in T$, that is, $t\, 0_T = 0_T = 0_T\, t$ for all $t \in T$.

\begin{lemma}\label{lem:absorber}
An element $e \in M$ is an absorber if and only if $t \triangleright_e m = t \cdot m$ for all $t \in T$ and $m \in M$. In particular, a $T$-module $(M, \cdot)$ admits an absorber $e$ if and only if $(M,\cdot) = (M, \triangleright_e)$.
\end{lemma}

\begin{proof}
Since $e$ is always an absorber for the $e$-induced action, if $t \triangleright_e m = t \cdot m$, then $e$ is an absorber in $M$. Conversely, if $e$ is an absorber in $M$ then
\begin{equation*}
t \triangleright_e m = [t \cdot m, t \cdot e, e] = [t \cdot m,e,e] = t\cdot m.\qedhere
\end{equation*}
\end{proof}

The notion of a stabilizer is complemented by that of an annihilator. Unlike the former, the latter is associated to an element and its nature depends on the nature of this element.

{
\begin{lemma}\label{lem:annMod}
Let $T$ be a truss and let $M$ be a non-empty $T$-module. 
For any $e \in M$, if the set
\[
\mathrm{Ann}_e(M) \coloneqq \big\{z\in T \mid z \cdot m = e, \textrm{ for all } m\in M\big\}
\]
is non-empty, then $\mathrm{Ann}_e(M)$ is a paragon and a right ideal in $T$. If, moreover, $e \in M$ is an absorber, then $\mathrm{Ann}_e(M)$ is a two-sided ideal in $T$.
\end{lemma}

\begin{proof} 
For every $e \in M$ consider the constant map $0_e\colon M \lra M, m \lto e,$ in $\eE(M)$ and the truss homomorphism $\phi:T\lra \eE(M)$, $\phi(t)(m)=t\cdot m$. If $\mathrm{Ann}_e(M) \not=\varnothing$, then $0_e \in \phi(T)$,  $\mathrm{Ann}_e(M)$ is the inverse image of $0_e$ and hence it is a paragon by \cite[Lemma 3.21]{Brz:par}. Since $0_e\circ f = 0_e$, i.e.\ $0_e$ is a left absorber in $\eE(M)$ in the terminology of \cite[Remark~3.13]{Brz:par}, $\mathrm{Ann}_e(M)$ is also a right ideal (adapt \cite[Lemma 3.27]{Brz:par}). If $e \in M$ is also an absorber in $M$, then $0_e$ is an absorber in $\phi(T) \subseteq \eE(M)$ and hence $\mathrm{Ann}_e(M)$ is a two-sided ideal by \cite[Lemma 3.27]{Brz:par}.
\end{proof}

\begin{definition}\label{def:annMod}
Let $T$ be a truss and let $M$ be a $T$-module.
For any $e \in M$, the set $\mathrm{Ann}_e(M)$ is called the \emph{$e$-annihilator} or \emph{$e$-contracting paragon} of $M$.

A $T$-module $M$ is said to be {\em $e$-contractible} if $\mathrm{Ann}_e(M)$ is non-empty.
\end{definition}
}

{
To conclude the subsection and in analogy with Examples \ref{ex:heaps} and \ref{ex:truss}, the subsequent results are aimed at showing that modules over trusses can be understood as equivalence classes of congruences on modules over rings.

\begin{lemma}\label{lem:EquiClIFFIndSubm}
Let $R$ be a ring and let $\tT(R)$ be the associated truss. A subset $S$ of an $R$-module $M$ is an equivalence class of a congruence modulo an $R$-submodule $N$ if and only if it is an induced $\tT(R)$-submodule of the $\tT(R)$-module $\tT(M)$ as in \cite[\S4]{BrzRyb:mod} or Example \ref{ex:modules}\eqref{ex:assomodule}.
\end{lemma}

\begin{proof}
If $S = m + N$ for some $m \in M$ and some $R$-submodule $N \subseteq M$, then we already know that $S$ is a sub-heap of $\hH(M)$ and moreover
\[r \triangleright_m(m + n) = \big[r\cdot (m+n),r\cdot m,m\big] = r\cdot m + r\cdot n -  r\cdot m + m = m + r\cdot n \in m + N,\]
for all $n \in N$, whence $R\triangleright_m S \subseteq S$ and thus $S$ is an induced submodule. 

Conversely, if $S \subseteq \tT(M)$ is an induced submodule, then for an arbitrary $s\in S$ we can consider $N \coloneqq \tau_s^0(S) = S - s$. We claim that $N \subseteq M$ is an $R$-submodule. In fact, 
\[r\cdot (s'- s) = r\cdot s' - r\cdot s + s - s = r \triangleright_s s' - s \in S - s\]
for all $ r\in R$ and hence $r \cdot N \subseteq N$. The assertion follows by observing that $S = s + N$.
\end{proof}
}

\begin{proposition}
Given a truss $T$, any non-empty $T$-module can be realised as an equivalence class of a congruence modulo a submodule in a module over a ring.
\end{proposition}

\begin{proof}
Let $M$ be a non-empty $T$-module. We already know from Example \ref{ex:heaps}\ref{item:exheaps3} that $\mathrm{Gr}_*(M) = \gG(M \boxplus \star; *)$ is an abelian group. Denote it by $\rR(M)$. \label{rR(M)} By \cite[\S3]{BrzRyb:mod}, $M \boxplus \star$ is a $T$-module as follows.
For every $t \in T$, we can consider morphisms of heaps
\[M \lra M \boxplus \star, \quad m \lto t \cdot m, \qquad \text{and} \qquad \star \lra M \boxplus \star, \quad * \lto *.\]
By the universal property of the coproduct, there exists a unique morphism of heaps $\lambda_t\colon M \boxplus \star \lto M \boxplus \star$ extending them. Furthermore, since for all $r,s,t \in T$ the morphisms of heaps $\lambda_{[t,r,s]}$ and $\big[\lambda_t,\lambda_r,\lambda_s\big]$ coincide on $M$ and on $\star$, they coincide on $M \boxtimes \star$, and analogously for the morphisms $\lambda_{ts}$ and $\lambda_t \circ \lambda_s$. Therefore, we constructed a morphism of trusses 
\[T \lra \ahrd({M \boxplus \star}),\]
which makes of $M \boxplus \star$ a $T$-module. Proceeding further, we can consider the obvious $\star$-modules structure 
\[\star \lra \ahrd({M \boxplus \star}), \qquad * \lto [x \lto *]\]
and hence the heap (truss, in fact) homomorphism
\begin{equation}\label{eq:modring}
T \boxplus \star  \lra \ahrd({M \boxplus \star})
\end{equation}
given by the universal property of the coproduct. Denote by $\rR(T)$ the abelian group $\gG(T \boxplus \star)$ with the ring structure coming from Example \ref{ex:truss}. Since every endomorphism in the image of \eqref{eq:modring} preserves the neutral element $*$, it follows from Corollary \ref{cor:mapsheaps} that we constructed a ring homomorphism
\[\rR(T) \lra \Ab({\rR(M)}),\]
so that $\rR(M)$ is an $\rR(T)$-module. 
The canonical map $M \to M \boxtimes \star$ of the coproduct allows us to realise $M$ as a sub-heap of $\hH(\rR(M))$ and since
\[t \triangleright_e m = [t\cdot m,t\cdot e,e]_{\rR(M)} = [t\act m,t\act e,e]_{M} \in M,\]
where subscripts indicate in which set the heap operations are taken, and 
\[* \triangleright_e m = [* \cdot m,* \cdot e,e]_{\rR(M)} = [*,*,e]_{\rR(M)} = e \in M\]
for all $m,e \in M$, $t \in T$, it follows that $M$ is an induced $\tT(\rR(T))$-submodule of $\rR(M)$ and hence an equivalence class of a congruence modulo a submodule ($(-m)+M$, in fact) in the $\rR(T)$-module $\rR(M)$.
\end{proof}

\subsection{Abelian groups with a \texorpdfstring{$T$}{T}-module structure}\label{ssec:Tgrps}

Let $\MAbs$ denote the category whose objects are pairs, a $T$-module and a fixed absorber, and morphisms are $T$-linear maps that preserve the absorbers. \label{p-t-abs}
Note that all the modules in $\MAbs$ are inhabited (they are non-empty). Our aim is to interpret $\MAbs$ as the category of abelian groups with $T$-actions.

\begin{proposition}
The category $\MAbs$ is (isomorphic to) the under category $\star/T\Mod$.
\end{proposition}

\begin{proof}
Recall that $\star=\{*\}$ is a $T$-module with $[*,*,*] = *$ and $t \cdot * = *$ for all $t \in T$. In light of this, for every $T$-module $M$ we have
\[T\Mod(\star,M) = \abs{M} \coloneqq \{m \in M \mid t \cdot m = m \textrm{ for all } t \in T\}.\]
Keeping in mind the foregoing observations, the property is almost tautological. To every object $(M,\cdot,0_M)$ in $\MAbs$ we assign the object $\big((M,\cdot),\star \lra M\colon  * \lto 0_M\big)$ in $\star/T\Mod$ and to every morphism $f \colon  M \lra N$ in $\MAbs$ we assign the morphism $f \colon  M \lra N$ itself in $\star/T\Mod$, and conversely. This gives the desired isomorphism.
\end{proof}

\begin{definition}
Let $T$ be a truss. A {\it $T$-group} is an abelian group $G$ together with an action \mbox{$\cdot\colon T\times G\to G$} of the multiplicative semigroup of $T$ on $G$ such that for all $t,t',t''\in T$ and $g,h\in G$, 
\begin{equation}\label{eq:TGrp}
[t,t',t'']\cdot g = t\cdot g - t'\cdot g + t''\cdot g \qquad  \text{and} \qquad t\cdot (g+h) = t\cdot g+t\cdot h.
\end{equation}
If additionally, there exists $t\in T$ such that $t\cdot g=g$ for all $g\in G$, then we say that $G$ is an {\it isotropic $T$-group}.

A morphism of $T$-groups is by definition a group homomorphism $f \colon  G \to G'$ such that 
$f(t \cdot g) = t \cdot f(g)$
for all $g \in G$ and $t \in T$. For the sake of brevity, we may often call them \emph{$T$-linear group homomorphisms}.

All $T$-groups together with $T$-linear group homomorphisms form the category $T\GMod$. The full subcategory of isotropic $T$-groups is denoted by $T\GMod_{\mathbf{is}}$.\label{p-t-gmod}
\end{definition}

\begin{remark}
An abelian group $G$ is a $T$-group if and only if $\hH(G)$ is a $T$-module and $0_G$ is an absorber.
\end{remark}

\begin{example}\label{ex:easy}
Let us present two elementary examples which will be useful later on.
\begin{enumerate}[leftmargin=0.8cm,ref=(\arabic*)]
\item Let $R$ be a ring. Then any $R$-module $M$ is a $\tT(R)$-group. In particular, all vector spaces over a field $\FF$ are $\tT(\FF)$-groups.
\item\label{item:easy2} Every group $G$ is an isotropic $\star$-group with action $*\cdot g=g$ for all $g\in G$.
\end{enumerate}
\end{example}

\begin{theorem}\label{thm:Tgroups}
The category $\MAbs$ is (isomorphic to) the category $T\GMod$. 
\end{theorem}

\begin{proof}
We know from Remark \ref{rem:abelianization} that the under category $\star/\Ah$ is isomorphic to $\Ab$. This isomorphism restricts to an isomorphism between $\star/T\Mod$ and $T\GMod$.
\end{proof}

Henceforth, we will often omit the $\cdot$ symbol when denoting the action of a truss $T$ on a $T$-module $M$, that is, we will simply write $tm$ instead of $t \cdot m$.

\section{Heaps of modules}\label{sec:mod}
In this section we introduce and study heaps of $T$-modules, where $T$ is a truss.

                                        
\subsection{Heaps of modules. First properties}\label{subs:first}

\begin{definition}[Heaps of modules]\label{def:hom}
Let $T$ be a truss. A {\em heap of $T$-modules} $(M,[-,-,-],\Lambda)$ is an abelian heap $(M,[-,-,-])$ together with
\[\Lambda \colon  T \times M \times M \lra M, \qquad (t,m,n) \lto \Lambda(t,m,n),\]
such that
\begin{enumerate}[label=(\arabic*),ref=(\arabic*),leftmargin=0.8cm]
\item\label{item:HM1} $\Lambda$ is a heap homomorphism in the first and the third entry separately, that is,
\begin{equation}\label{eq:leftheap}
\Lambda\Big([t,t',t''],m,n\Big) = \Big[\Lambda\big(t,m,n\big),\Lambda\big(t',m,n\big),\Lambda\big(t'',m,n\big)\Big]
\end{equation}
and
\begin{equation}\label{eq:rightheap}
\Lambda\Big(t,m,[n,n',n'']\Big) = \Big[\Lambda\big(t,m,n\big),\Lambda\big(t,m,n'\big),\Lambda\big(t,m,n''\big)\Big]
\end{equation}
for all $t,t',t'' \in T$ and $m,n,n',n'' \in M$
\item\label{item:HM2} $\Lambda$ is $T$-associative, that is to say
\begin{equation}\label{eq:ass}
\Lambda(st,m,n) = \Lambda\Big(s,m,\Lambda\big(t,m,n\big)\Big)
\end{equation}
for all $s,t \in T$, $ m,n \in M$.
\item\label{item:HM3} $\Lambda$ satisfies the \emph{base change property}
\begin{equation}\label{eq:basechange}
\Lambda\big(t,m,n\big) = \Big[\Lambda(t,e,n),\Lambda(t,e,m),m\Big]
\end{equation}
for all $m,n,e \in M$, $t \in T$.
\end{enumerate}
A {\em morphism of heaps of modules} is a morphism $f \colon  M \lra N$ of heaps such that
\begin{equation}\label{eq:morph}
f\Big(\Lambda_M(t,m,n)\Big) = \Lambda_N\Big(t,f(m),f(n)\Big),
\end{equation}
for all $m,n \in M$, $t \in T$. The category of heaps of $T$-modules and their morphisms will be denoted by $T\HMod$.\label{p-t-hmod}
\end{definition}

\begin{example}\label{ex:trivial}
Let $T$ be a truss. Any abelian heap $H$ is a heap of $T$-modules with any of the two trivial actions:
\begin{blist}
\item $\Lambda(t,h,h') = h$,
\item $\Lambda(t,h,h') = h'$,
\end{blist}
for all $t\in T$ and $h,h'\in H$.
In fact, in both cases, conditions \eqref{eq:leftheap}, \eqref{eq:rightheap}, \eqref{eq:basechange} follow by the Mal'cev identities. The associativity \eqref{eq:ass} is obvious.
\end{example}

One may wonder if the base change property is independent or not of the other axioms. The following example shows that it is.

\begin{example}
Let $T=\tT(\ZZ)$ and $M=\hH(\ZZ)$, then we can define $${\Lambda\colon T\times M\times M\to M}\qquad \text{by}\qquad (a,p,m)\lto\begin{cases}
p\ \text{if}\ p\ \text{even}\\
m\ \text{if}\ p\ \text{odd}
\end{cases}.$$
One can easily show that $\Lambda$ fulfils all conditions from Definition \ref{def:hom} apart from the base change property \eqref{eq:basechange}. \note{Observe that, in this case, the translation map $\tau_e^f\colon M\to M,$ $m\lto [m,e,f]=m+e-f$ is not a homomorphism of heaps of modules.
Indeed,
\[
\tau_2^1(\Lambda(3,2,1))=\tau_2^1(2)=2-2+1 \neq 0 =\Lambda(3,1,0)=\Lambda(3,\tau_2^1(2),\tau_2^1(1)).
\]}
\end{example}

The following three lemmas explore the significance of the axioms presented in Definition~\ref{def:hom}, in particular they add significance to the base change property \eqref{eq:basechange} and show how it can be used to repair the seemingly asymmetric requirement for $\Lambda$ to be a heap morphism in the first and the third arguments, but not in the middle one.

\begin{lemma}
The base change property \eqref{eq:basechange} is equivalent to
\begin{equation}
\big[\Lambda(t,m,e),e,\Lambda(t,e,n)\big] = \Lambda(t,m,n) \label{eq:interchangebis}
\end{equation}
for all $t \in T$ and $m,n,e \in M$.
\end{lemma}

\begin{proof}
In view of the Mal'cev identities and of the abelianity of the bracket, 
\[
\Lambda(t,m,n) = \big[\Lambda(t,e,n),\Lambda(t,e,m),m\big]
\]
if and only if 
\[
\big[\Lambda(t,e,m),m,\Lambda(t,m,n)\big] = \Lambda(t,e,n),
\]
that is to say, the base change property \eqref{eq:basechange} is equivalent to \eqref{eq:interchangebis}.
\end{proof}

\begin{lemma}\label{lem:idem-inter}
 For fixed $t\in T$ and $n\in M$, the map $\Lambda(t,-,n)\colon M\to M$ is a heap homomorphism. Moreover, 
\begin{equation}\label{eq:idempotent}
\Lambda(t,m,m)=m,
\end{equation}
\begin{equation}\label{eq:interchange}
\Lambda(t,m,n) = [n,\Lambda(t,n,m),m]
\end{equation}
and 
\begin{equation}\label{eq:negation}
\Lambda\big(t,[e,m,f],[e,n,f]\big) = [e,\Lambda(t,m,n),f],
\end{equation}
for all $t\in T$, $m,m',n,e,f \in M$.
\note{Consequently, 
\begin{equation}\label{eq:middle-comp}
\Lambda(t,[m_1,m_2,m_3],n) = [\Lambda(t,m_1,n),\Lambda(t,m_2,n),\Lambda(t,m_3,n)]
\end{equation}
for all $t\in T$ and $m_1,m_2,m_3,n\in M$.}
\end{lemma}

\begin{proof}
Observe that the equation \eqref{eq:basechange} entails that
\[
\begin{aligned}
\Lambda\Big(t,&[m,m',m''],n\Big)  = \Big[\Lambda(t,e,n),\Lambda(t,e,[m,m',m'']),[m,m',m'']\Big] \\
 & = \Big[\Lambda(t,e,n),\Lambda(t,e,m),\Lambda(t,e,m'),\Lambda(t,e,m''),m,m',m''\Big] \\
 & = \Big[\Lambda(t,e,n),\Lambda(t,e,m),m, \Lambda(t,e,n),\Lambda(t,e,m'),  m' ,\Lambda(t,e,n),\Lambda(t,e,m''),m''\Big]\\
 & = \Big[\Lambda(t,m,n),\Lambda(t,m',n),\Lambda(t,m'',n)\Big],
\end{aligned}
\]
that is to say, that $\Lambda$ is a heap map in the middle entry. Furthermore,
\[\Lambda(t,m,m) = \left[\Lambda(t,e,m),\Lambda(t,e,m),m\right] = m\]
for all $t\in T$, $e,m \in M$, while equation \eqref{eq:interchange} follows by replacing $e$ on the right hand side of \eqref{eq:basechange} by $n$ and using \eqref{eq:idempotent}. 

 Finally, by using \eqref{eq:interchangebis} 
 (to derive the third and fifth equalities) as well as \eqref{eq:idempotent} (to derive the second equality)  
 and the fact that $\Lambda$ is a heap morphism in all three arguments (to  derive the first equality), we can compute
 $$
 \begin{aligned}
 \Lambda \big(t,[e,m,f],[e,n,f]\big) =&\Big[\Lambda(t,e,e),\Lambda(t,e,n), \Lambda(t,e,f),
 \Lambda(t,m,e),\Lambda(t,m,n),\\
 &\Lambda(t,m,f),
 \Lambda(t,f,e),\Lambda(t,f,n), \Lambda(t,f,f) \Big]\\
 =& \Big[e,\Lambda(t,e,n), \Lambda(t,e,f), f, \Lambda(t,f,e),
 \Lambda(t,m,e),\Lambda(t,m,n),\\
 & \Lambda(t,m,f), f, \Lambda(t,f,n), f \Big]\\
 =& \Big[e,\Lambda(t,e,n), e,
 \Lambda(t,m,e),\Lambda(t,m,n),
 \Lambda(t,m,n),f \Big]\\
 =& \Big[e,\Lambda(t,m,n),  f \Big].
 \end{aligned}
 $$
 In addition to the above mentioned properties, we have also freely used the Mal'cev identities and the reshuffling rules for abelian heaps. This proves \eqref{eq:negation}.
\end{proof}

We note in passing that the property \eqref{eq:negation} implies that in $\gG(M;e)$ we have
$$
\Lambda(t,-m,-n) = -\Lambda(t,m,n),
$$
for all $m,n\in M$ and $t\in T$, where we recall that $-m = [e,m,e]$. Furthermore, from \eqref{eq:idempotent} it follows that all constant maps are morphisms of heaps of modules, as the following example exhibits.

\begin{example}\label{ex:constmaps}
Let $T$ be a truss and let $M$ be a heap of $T$-modules. For any $m \in M$, the constant function 
\[
\widehat{m}\colon M \lra M, \qquad x \lto m,
\]
is a morphism of heaps of $T$-modules. In fact, it is clearly a morphism of heaps and the compatibility with $\Lambda$ follows from \eqref{eq:idempotent}.
\end{example}

\begin{lemma}\label{lem:trans}
Let $M$ be a heap with a function $\Lambda\colon  T\times M\times M \lra M$ that satisfies \eqref{eq:leftheap}, \eqref{eq:rightheap}, \eqref{eq:ass} and \eqref{eq:idempotent}. Then, for all $e,f\in M$, the map 
$$\tau_e^f\colon M\lra M, \qquad  m\lto [m,e,f],
$$
is a morphism of heaps of modules (that is, $\tau_e^f$ satisfies \eqref{eq:morph}) if and only if the base change property \eqref{eq:basechange} holds.
\end{lemma}
\begin{proof}
Let $M$ satisfy \eqref{eq:leftheap}, \eqref{eq:rightheap}, \eqref{eq:ass} and \eqref{eq:idempotent}.  Then, for all $a\in T$ and $m,e,f\in M$,
\[\tau_e^f(\Lambda(a,e,m))=[\Lambda(a,e,m),e,f]\]
equals
\[\Lambda(a,\tau_e^f(e),\tau_e^f(m))=\Lambda(a,f,[m,e,f])=[\Lambda(a,f,m),\Lambda(a,f,e),f]\]
if and only if
\[\Lambda(a,e,m)=[\Lambda(a,e,m),e,f,f,e]\]
equals
\[ [\Lambda(a,f,m),\Lambda(a,f,e),f,f,e] = [\Lambda(a,f,m),\Lambda(a,f,e),e] .\]
Summing up: the translation isomorphisms are morphisms of heaps of modules if and only if the base change property holds.
\end{proof}

Lemma~\ref{lem:idem-inter} can be used to characterise congruences of heaps of modules. Let $(M,\Lambda)$ be a heap of $T$-modules. By a {\em sub-heap of modules} we mean a sub-heap $N$ of $M$ that is closed under the operation $\Lambda$, that is, a sub-heap such that, for all $n,n'\in M$ and $t\in T$, $\Lambda(t,n,n')\in N$.

\begin{proposition}\label{prop:cong}
Let $(M,\Lambda)$ be a heap of $T$-modules.
\begin{enumerate}[label=(\arabic*),ref=(\arabic*),leftmargin=0.8cm]
\item\label{item:cong1} For every non-empty sub-heap of modules $N$ of $M$, the sub-heap relation $\sim_N$ is a congruence on $M$.
\item\label{item:cong2} If $\sim$ is a congruence on the heap of $T$-modules $M$, then its equivalence classes are sub-heaps of modules of $T$.
\end{enumerate}
\end{proposition}

\begin{proof}
\ref{item:cong1} Let $N$ be a non-empty sub-heap of modules and consider any $m\sim_N m''$. This means that there exist $n,n'\in N$ such that $m'' = [n,n',m]$. For all $m'\in M$,
$$
\begin{aligned}
\Lambda(t,m',m'') &= \Lambda(t,m', [n,n',m]) \\
&= [\Lambda(t,m', n),\Lambda(t,m', n'), \Lambda(t,m',m)]\\
&= [[\Lambda(t,m', n),\Lambda(t,m', n'), n'], n', \Lambda(t,m',m)]\\
&= [\Lambda(t,n', n),n', \Lambda(t,m',m)],
\end{aligned}
$$
by \eqref{eq:rightheap} and the base change property \eqref{eq:basechange} combined with the associativity of the heap operation and the Mal'cev identity. Hence
\begin{equation}\label{cong.1}
\Lambda(t,m',m'') \sim_N \Lambda(t,m',m).
\end{equation}
Next, by using in addition equation \eqref{eq:interchange} in Lemma~\ref{lem:idem-inter} and \eqref{eq:interchangebis}, we can compute
\note{$$
\begin{aligned}
\Lambda(t,m'',m') &= [m',\Lambda(t,m',m''),m'']\\
&= [m',\Lambda(t,m',m),\Lambda(t,m', n'), \Lambda(t,m', n),n,n',m]\\
&= [m',\Lambda(t,m',m), \Lambda(t,n,n'),n', m]\\
&= [m',\Lambda(t,m',m), m, n', \Lambda(t,n,n')] = [\Lambda(t,m,m'), n', \Lambda(t,n,n')].
\end{aligned}
$$}
\[
\begin{aligned}
\Lambda(t,m'',m') & = \big[\Lambda(t,n,m'),\Lambda(t,n',m'),\Lambda(t,m,m')\big] \\
& = \big[\Lambda(t,n,m'),m',\Lambda(t,m',n'),n',\Lambda(t,m,m')\big] \\
& = \big[\Lambda(t,n,n'),n',\Lambda(t,m,m')\big].
\end{aligned}
\]
Therefore,
\begin{equation}\label{cong.2}
\Lambda(t,m'',m') \sim_N \Lambda(t,m,m').
\end{equation}
Finally, using \eqref{cong.1} and \eqref{cong.2} and the transitivity of equivalence relations we conclude that in general, if $m\sim_N m''$ and $m'\sim_N m'''$, then $\Lambda(t,m'',m''') \sim_N \Lambda(t,m,m')$, as required for a congruence.

\ref{item:cong2} Let $[m]$ be a congruence class of $m\in M$. Then $[m]$ is a sub-heap of $M$ by the Mal'cev identities (idempotency of the heap operation). Similarly, \eqref{eq:idempotent} in Lemma~\ref{lem:idem-inter} implies that if $m'\sim m$ and $m''\sim m$, then, for all $t\in T$,
$$
\Lambda(t,m',m'') \sim \Lambda(t,m,m) =m,
$$ 
and hence $[m]$ is a sub-heap of modules.
\end{proof}

The next two results justify the claim that heaps of modules can be understood as a different point of view on induced (sub)modules.  

\begin{proposition}\label{prop:homia}
Let $T$ be a truss, $M$ be a $T$-module and let $N$ be an induced $T$-submodule of $M$. Then $(N,\triangleright)$ is a heap of $T$-modules, where $\triangleright\colon T\times N\times N \to N$ is a map which assigns to every triple $(t,e,n)\in T\times N\times N$, the induced $e$-action $t\triangleright_e n$. Furthermore, the assignment
\[\Hh\colon T\Mod \to T\HMod, \qquad (M,[-,-,-],\cdot) \lto (M,[-,-,-],\triangleright)\]
is functorial.
\end{proposition}

\begin{proof}
Observe that since $N$ with the induced action $\triangleright_e$ is a $T$-module, for all $e\in N$, it is enough to check that the base change property holds. Let $m,n \in N$, $t \in T$, then
$$
[t\triangleright_e n, t\triangleright_e m, m]=[t\!\cdot\! n,t\!\cdot\!e,e,t\!\cdot\!m,t\!\cdot\!e,e,m]=[t\!\cdot\!n,t\!\cdot\!m,m] =t\triangleright_m n.
$$
Thus $(N,\triangleright)$ is a heap of $T$-modules. Moreover, if $f\colon  M \lra M'$ is a morphism of $T$-modules then 
\[f\left(t \triangleright_e m\right) = f\big([t\!\cdot\!m, t \!\cdot\!e,e]\big) = \big[t\!\cdot\!f(m), t\!\cdot\!f(e),f(e)\big] = t \triangleright_{f(e)} f(m)\]
and hence $f$ is a morphism of heaps of $T$-modules as well.
\end{proof}
\begin{example}
Since an abelian heap $H$ is an $\eE(H)$-module by evaluation (see Example~\ref{ex:modules}), the map 
$
\Lambda:\eE(H)\times H\times H\lra H,$  $(f,a,b) \lto [f(b),f(a),a],$
makes $H$ into a heap of $\eE(H)$-modules.
\end{example}

\begin{example}\label{ex:E(T)}
Let $T$ be a truss. The endomorphism truss $\eE(T)$ of $T$ as an abelian heap is a $T$-module with $(t \cdot f)(t') \coloneqq tf(t')$ for all $t,t'\in T$ and $f \in \eE(T)$. Therefore, it also enjoys a structure of heap of $T$-modules as in Proposition \ref{prop:homia}, explicitly given by $\Lambda(t,f,f') = [t\cdot f', t\cdot f, f]$.
\end{example}

\begin{lemma}\label{lem:homtoind}
Let $(M,\Lambda)$ be a heap of $T$-modules over a truss $T$. Then for any chosen $e\in M$, the pair $\big(M,\act_e\big)$, where  $\act_e \coloneqq \Lambda(-,e,-)\colon T\times M\to M$, is a $T$-module with absorber $e$. Moreover $\Hh\big(M,\act_e\big) = (M,\Lambda)$ for every heap of $T$-modules $(M,\Lambda)$ and $\big(\Hh(M,\cdot),\act_e\big) = (M,\triangleright_e)$ for every $T$-module $(M,\cdot)$.
\end{lemma}

\begin{proof}
Observe that since $\Lambda$ is a heap homomorphism on the first and third entries and it is $T$-associative, $\Lambda(-,e,-)$ is a $T$-action on $M$. Moreover, $\Lambda(t,e,e)=e$, for all $t\in T$, by \eqref{eq:idempotent}. Thus $e$ is an absorber for $\Lambda(-,e,-)$. If we consider $\triangleright\colon T \times M \times M \to M$ associated with $\act_e = \Lambda(-,e,-)$ as in Proposition \ref{prop:homia}, then 
\[t\triangleright_m n = \big[t\act_e n, t \act _e m,m\big] = \Big[\Lambda(t,e,n),\Lambda(t,e,m),m\Big] \stackrel{\eqref{eq:basechange}}{=} \Lambda(t,m,n)\]
for any $t\in T$ and $m,n\in M$. Therefore, $\triangleright=\Lambda$ and $\Hh\big(M,\Lambda(-,e,-)\big)=(M,\Lambda)$. Finally, in $\big(\Hh(M,\cdot), \act_e\big)$, for all $m \in M$ and $t \in T$,
\[t \act_e m = \triangleright\big(t,e,m\big) = t \triangleright_e m\]
and hence $\big(\Hh(M,\cdot);e\big) = (M,\triangleright_e)$.
\end{proof}

In parallel with heaps and groups, the construction assigning $(M,\Lambda)$ to $\big(M,\act_e\big)$ from Lemma \ref{lem:homtoind} is not functorial, as it depends on an arbitrary choice.

Put together, Proposition~\ref{prop:homia} and Lemma~\ref{lem:homtoind} immediately yield
\begin{corollary}\label{cor:ind}
Every heap of $T$-modules $(M,\Lambda)$ is the heap of $T$-modules $\Hh\big(M,\triangleright_e\big)$ associated with an induced action $\triangleright_e$ on some $T$-module $M$ (with absorber).
\end{corollary}

\begin{remark}\label{rem:nec?}
It is worth to mention here that two different $T$-modules can have the same induced module structure. Consider a $T$-module $(M,\cdot)$ without an absorber and its induced $T$-module $(M,\triangleright_e)$, for some $e\in M$. Then $((M,\cdot),\triangleright_e)=(M,\triangleright_e)$, but $(M,\triangleright_e)$ is not even isomorphic with $(M,\cdot)$ as there is no $T$-module homomorphism from $(M,\triangleright_e)$ to $(M,\cdot)$.
\end{remark}

The correspondence of Corollary~\ref{cor:ind} can be used to prove the following entropic or interchange property of heaps of modules.
\begin{lemma}\label{lem:entropy}
Let $(M,\Lambda)$ be a heap of $T$-modules. Let $t,t'\in T$ be such that for some $e\in T$ and all $m\in M$,
\begin{equation}\label{commut}
\Lambda(tt',e,m) = \Lambda (t't,e,m).
\end{equation}
Then, for all $m,m', m'', n\in M$,
\[
\Lambda\Big(t, \Lambda\left(t',m,m'\right), \Lambda\left(t',m'',n\right)\Big) = \Lambda\Big(t', \Lambda\left(t,m,m''\right), \Lambda\left(t,m',n\right)\Big).
\]
\end{lemma}
\begin{proof}
With no loss of generality we may assume that $\Lambda$ is given by an induced action on a $T$-module $M$ (with absorber $a$), that is, 
$$
\Lambda(t,m,n) = t \triangleright_mn = [t\!\cdot\!n,t\!\cdot\!m,m].
$$
Furthermore, the base change property implies that if the equality \eqref{commut} holds for one $e\in M$, then it does so for all $e\in M$. Thus, in particular,
\begin{equation}\label{comm.m}
tt'\!\cdot\! m = [tt'\!\cdot\!m,tt'\!\cdot\!a,a] = \Lambda(tt',a,m) = \Lambda(t't,a,m) = [t't\!\cdot\!m,t't\!\cdot\!a,a] = t't\!\cdot\!m,
\end{equation}
for all $m\in M$. 
Therefore, using the distributive laws of the $T$-action to derive the first and third equalities and the fact that $M$ is an abelian heap and \eqref{comm.m} to derive the middle one, we can compute,
$$
\begin{aligned}
\Lambda\Big(t, \Lambda\left(t',m,m'\right), &\; \Lambda\left(t',m'',n\right)\Big)\\ &= 
\Big[[tt'\!\cdot\!n,tt'\!\cdot\!m'',t\!\cdot\!m''],[tt'\!\cdot\!m',tt'\!\cdot\!m,t\!\cdot\!m],[t'\!\cdot\!m',t'\!\cdot\!m,m]\Big]\\
&= \Big[[t't\!\cdot\!n,t't\!\cdot\!m',t'\!\cdot\!m'],[t't\!\cdot\!m'',t't\!\cdot\!m,t'\!\cdot\!m],[t\!\cdot\!m'',t\!\cdot\!m,m]\Big]\\
&= \Lambda\Big(t', \Lambda\left(t,m,m''\right), \Lambda\left(t,m',n\right)\Big),
\end{aligned}
$$
as required.
\end{proof}

\begin{remark}
Recall that if $M,N$ are abelian heaps, then we may perform the tensor product $M \otimes N$ of abelian heaps and this satisfies properties similar to those of the tensor product of modules. With this convention and in view of Lemma \ref{lem:homtoind}, a heap of $T$-modules is an abelian heap $(M,[-,-,-])$ together with a heap homomorphism
\[\Lambda \colon  T \otimes M \otimes M \to M, \qquad t \otimes m \otimes n \lto t \triangleright_mn\]
which is $T$-associative
\[ts \triangleright_m n = t \triangleright_m \Big(s \triangleright_m n \Big)\]
and satisfies the base change property
\[t \triangleright_m n = \big[t\triangleright_en,t \triangleright_em,m\big]\]
for all $m,n,e \in M$ and $t,s \in T$. 
\end{remark}

{

\subsection{Isotropy and contractibility for heaps of modules}\label{ssec:UZprop}

Recall from Lemma \ref{lem:homtoind} that if $M$ is a heap of $T$-modules, then $(M,\cdot_e)$ denotes the associated $T$-module with $t\cdot_em = \Lambda(t,e,m)$ for all $t \in T$, $e,m \in M$.

As modules over trusses do not need to be isotropic or contractible in general, the same happens for heaps of modules. Therefore, we are led to the following lemmata and definitions.

\begin{lemma}\label{lem:isotropicheap}
Let $T$ be a truss. For any heap of $T$-modules $M$, the set
\[\mathrm{Stab}(M) \coloneqq \{u\in T\; |\; \Lambda(u,m,n) = n, \; \mbox{for all $m,n\in M$}\}\]
is a sub-truss and, if non-empty, a paragon of $T$.
\end{lemma}

\begin{proof}
First of all, notice that if $M = \varnothing$, then $\mathrm{Stab}(M) = T$ and the statement is true. If $M\neq \varnothing$, then for any $e \in M$, $\mathrm{Stab}(M) = \mathrm{Stab}(M,\act_e)$, by the base change property \eqref{eq:basechange}, where
\[\mathrm{Stab}(M,\act_e) = \{u\in T\; |\; \Lambda(u,e,n) = n, \; \mbox{for all $n\in M$}\},\]
(see Lemma \ref{lem:stab}). 
The inclusion $\mathrm{Stab}(M) \subseteq \mathrm{Stab}(M,\act_e)$ holds trivially, while, for every $u \in \mathrm{Stab}(M,\act_e)$,
\[\Lambda(u,m,n) \stackrel{\eqref{eq:basechange}}{=} \big[\Lambda(u,e,n),\Lambda(u,e,m),m\big] = [n,m,m] \stackrel{\eqref{eq:malcev}}{=} n\]
for all $n \in M$ and hence $\mathrm{Stab}(M,\act_e) \subseteq \mathrm{Stab}(M)$, too. Therefore, the statement follows by Lemma \ref{lem:stab}.
\end{proof}

\begin{definition}
For a heap of $T$-modules $M$,
$\mathrm{Stab}(M)$ is called the \emph{stabilizer} or \emph{isotropy paragon} of $M$.
~
A heap of $T$-modules $M$ is said to be {\em isotropic} if $\mathrm{Stab}(M)$ is non-empty.
~
The category of {\em isotropic heaps of $T$-modules} $T\HMod_{\mathbf{is}}$ is the full subcategory of heaps of $T$-modules whose objects are isotropic heaps of $T$-modules. \label{p-isotropic-h}
\end{definition}

\begin{remark}
It follows from the proof of Lemma \ref{lem:isotropicheap} that a (non-empty) heap of $T$-modules $M$ is isotropic if and only if $\big(M,\act_e\big)$ is an isotropic $T$-module for all $e \in M$, if and only if there exists $e \in M$ such that $\big(M,\act_e\big)$ is isotropic.
\end{remark}

\begin{lemma}\label{lem:annihilator}
Let $T$ be a truss. For any heap of $T$-modules $M$, the set
\[\mathrm{Ann}(M) \coloneqq \big\{z\in T \mid \Lambda(z,m,n) = m, \; \mbox{for all $m,n\in M$}\big\},\]
 if non-empty, is a two-sided ideal in $T$.
\end{lemma}

\begin{proof}
Analogously to the proof of Lemma \ref{lem:isotropicheap}, if $M = \varnothing$, then $\mathrm{Ann}(M) = T$ and the statement is trivially true. If $M \neq \varnothing$, then for every $e \in M$ we have that $\mathrm{Ann}(M) = \mathrm{Ann}_e(M,\act_e)$ by the base change property \eqref{eq:basechange}, where
\[
\mathrm{Ann}_e(M,\act_e) = \big\{ z \in T \mid z \act_e m = e \textrm{ for all } m \in M \big\}.
\]
The latter is a two-sided ideal by Lemma \ref{lem:annMod}.
\end{proof}

\begin{definition}
For a heap of $T$-modules $M$,
$\mathrm{Ann}(M)$ is called the \emph{annihilator} or \emph{contracting ideal} of $M$.
~
A heap of $T$-modules $M$ is said to be {\em contractible} if $\mathrm{Ann}(M)$ is non-empty.
~
The category of {\em contractible heaps of $T$-modules} $T\HMod^{\mathbf{cn}}$ is the full subcategory of heaps of $T$-modules whose objects are contractible heaps of $T$-modules.\label{p-contr-mod} 
\end{definition}

\begin{proposition}
Let $T$ be a non-empty truss. The following statements are equivalent for a non-empty heap of $T$-modules $M$:
\begin{enumerate}[label=(\arabic*),ref=(\arabic*),leftmargin=0.8cm]
\item\label{item:zero1} $M$ is contractible. 
\item\label{item:zero2} For every $e \in M$, the $T$-module $\big(M,\act_e\big)$ with absorber $e \in M$ is $e$-contractible.
\item\label{item:zero3} There exists  $e \in M$ such that $(M,\act_e)$ is $e$-contractible.
\end{enumerate}
\end{proposition}

\begin{proof}
The chain of implications \ref{item:zero1} $\Rightarrow$ \ref{item:zero2} $\Rightarrow$ \ref{item:zero3} is clear, while to prove that \ref{item:zero3} $\Rightarrow$ \ref{item:zero1} is enough to note that $\mathrm{Ann}(M) = \mathrm{Ann}_e(M,\act_e) \neq \varnothing$, as in the proof of Lemma \ref{lem:annihilator}.
\end{proof}

\begin{theorem}\label{thm:IUmods}
Let $T$ be a truss and let $M$ be a heap of $T$-modules.
\begin{enumerate}[label=(\arabic*),ref=(\arabic*),leftmargin=0.8cm]
    \item\label{item:IUmod2} 
    If $M$ is an isotropic heap of $T$-modules 
    then, by considering $T$ as a module over itself, $\mathrm{Stab}(T) \subseteq \mathrm{Stab}(M)$. In particular, if $T$ is a unital truss with unit $1_T$, then $1_T \in \mathrm{Stab}(M)$.
    \item\label{item:IUmod3} If $M$ is an isotropic heap of $T$-modules and $T$ is a unital truss, then the set
    \[\mathrm{Stab}(M)^\times \coloneqq \mathrm{Stab}(M) \cap T^\times = \left\{t \in T^\times \mid \Lambda(t,m,n) = n \textrm{ for all } m,n \in M\right\}\]
    is a group with respect to the product in $T$.
    \item\label{item:IUmod1} If $M$ is a contractible heap of $T$-modules and the truss $T$ admits an absorber $0_T$, then $0_T \in \mathrm{Ann}(M)$.
\end{enumerate}
\end{theorem}

\begin{proof}
\ref{item:IUmod2} If $u_T \in T$ is such that $u_T\, t=t$ for all $t \in T$ and $u_M \in T$ is such that $\Lambda(u_M,m,n) = n$ for all $m,n \in M$, then
\[\Lambda(u_T,m,n) = \Lambda\big(u_T,m,\Lambda(u_M,m,n)\big) \stackrel{\eqref{eq:ass}}{=} \Lambda(u_T\,u_M,m,n) = \Lambda(u_M,m,n) = n\]
for all $m,n\in M$ and hence $u_T \in \mathrm{Stab}(M)$.

\ref{item:IUmod3} In view of Lemma \ref{lem:isotropicheap} we know that $\mathrm{Stab}(M)^\times$ is closed under the product of $T$. In view of \ref{item:IUmod2}, we know that $1_T \in \mathrm{Stab}(M)^\times$. Finally, if $t \in \mathrm{Stab}(M)^\times$ then
\[\Lambda(t^{-1},m,n) = \Lambda\big(t^{-1},m,\Lambda(t,m,n)\big) \stackrel{\eqref{eq:ass}}{=} \Lambda(t^{-1}t,m,n) = \Lambda(1_T,m,n) = n\]
for all $m,n \in M$, whence $t^{-1} \in \mathrm{Stab}(M)^\times$ too.

\ref{item:IUmod1} Suppose that $z_M \in T$ satisfies $\Lambda(z_M,m,n) = m$, for all $m,n\in M$. Then
\[m \stackrel{\eqref{eq:idempotent}}{=} \Lambda(0_T,m,m) = \Lambda\big(0_T,m,\Lambda(z_M,m,n)\big) \stackrel{\eqref{eq:ass}}{=} \Lambda(0_T\,z_M,m,n) = \Lambda(0_T,m,n),\]
for all $m,n \in M$. Thus $0_T \in \mathrm{Ann}(M)$.
\end{proof}
}

\subsection{Crossed products by heaps of modules}

The following proposition extends the construction of \cite[Theorem~4.2]{BrzRyb:con} to heaps of modules.

\begin{proposition}\label{prop:cross}
Let $T$ be a non-empty truss and $(M,\Lambda)$ be a non empty heap of $T$-modules. Then, for all $e\in M$, the Cartesian product of heaps $M\times T$ is a truss with multiplication
\[
(m,s)(n,t) = ([\Lambda(s,e,n),e,m], st),
\]
for all $(m,s),(n,t) \in M\times T$. We denote this truss by $M\overset{e}{\rtimes} T$ and call it a {\em cross product} of $T$ with $M$.
\end{proposition}

\begin{proof}
By Lemma \ref{lem:homtoind}, if $(M,\Lambda)$ is a heap of $T$-modules and $a \in M$, then $(M,\act_a)$ is a $T$-module, where $t \act_a m = \Lambda(t,a,m)$ for all $t \in T, m\in M$. Therefore, by \cite[Theorem~4.2]{BrzRyb:con}, for any $e\in M$,  
$$
\begin{aligned}
(m,s)(n,t) &= \big([m,s\act_a e,s \act_a n],st\big) = \big([\Lambda(s,a,n),\Lambda(s,a,e),m],st\big) \\
 &= \big([\Lambda(s,a,n),\Lambda(s,a,e),e,e,m],st\big) \stackrel{\eqref{eq:basechange}}{=} \big([\Lambda(s,e,n),e,m],st\big)
\end{aligned}
$$
and $M \times T$ is a truss with the Cartesian product heap structure.
\end{proof}

The cross product by a heap of modules has similar properties to those of an extension of a truss by a module listed in \cite[Theorem~4.4]{BrzRyb:con}.

\begin{lemma}\label{lem:cross}
Let $T$ be a non-empty truss and $(M,\Lambda)$ be a non-empty heap of \mbox{$T$-modules.}
\begin{zlist}
\item $M$ is a left $M\overset{e}{\rtimes} T$-module by the action
$$
(m,t)\cdot n = [\Lambda(s,e,n),e,m].
$$
In particular $(m,t)\cdot e =m$, for all $m\in M$ and $t\in T$.
\item The trusses $M\overset{e}{\rtimes} T$ and $M\overset{f}{\rtimes} T$ are isomorphic for all $e,f\in M$.
\item For all $u\in T$, $M_u = M\times \{u\}$ is a paragon in $M\overset{e}{\rtimes} T$, and
$T \cong M\overset{e}{\rtimes} T/M_u$.
\item The heap $T_e =\{e\}\times T$ is a sub-truss and a left paragon of $M\overset{e}{\rtimes} T$, and  $M\cong M\overset{e}{\rtimes} T/T_e$ as left $M\overset{e}{\rtimes} T$-modules.
\end{zlist}
\end{lemma}

\begin{proof}
As in the proof of Proposition~\ref{prop:cross}, apply \cite[Theorem~4.4]{BrzRyb:con} to $T$ and the $T$-module $(M,\act_a)$ for $a \in M$.
\end{proof}


\section{Heaps of modules and affine spaces}\label{sec:aff}
In this section we will focus on a geometric interpretation of heaps of modules over a truss. In particular, we will highlight the relationship between heaps of $\tT(R)$-modules and affine modules over a ring $R$. We will show that the straightforward extension of the notion of an affine module over a ring $R$ to that of an affine space over a truss $T$ by changing an $R$-module to a $T$-group (and it will be clear soon why this is the natural choice) provides for us a category which turns out to be equivalent to the category of heaps of $T$-modules. Thus affine spaces over a ring $R$ or over a field $\mathbb{F}$ are special cases of heaps of $\tT(R)$ or $\tT(\mathbb{F})$-modules, respectively.

\subsection{Morphisms of heaps of modules as translations}\label{ssec:RHMod}
The first observation to be made is that all morphisms  of heaps of $T$-modules are translations of $T$-modules morphisms, for a particular choice of $T$-modules.

To this aim, recall from \S\ref{ssec:Tgrps} how the category of $T$-groups is isomorphic to the category of $T$-modules with a chosen absorber and morphisms preserving the absorbers.

\begin{proposition}\label{prop:newTHmodTMod}
Let $T$ be a truss and let $M,N$ be two non-empty heaps of $T$-modules. Let $m \in M$ and $n \in N$. Then a function $f \colon M \lra N$ is a morphism of $T$-groups from $\big(\gG(M;m),\act_m\big)$ to $\big(\gG(N;n),\act_n\big)$ if and only if $f$ is a morphism of heaps of $T$-modules such that $f(m) = n$.
\end{proposition}

\begin{proof}
If $f \in T\HMod(M,N)$ and $f(m) = n$, then $f\in \heap(M,N)$ and
\[f\big(t \act_m m'\big) = f\big(\Lambda_M(t,m,m')\big) = \Lambda_N\big(t,n,f(m')\big) = t \act_n f(m')\]
for all $m'\in M$, that is, it is a morphism of $T$-modules from $(M,\act_m)$ to $(N,\act_n)$. Moreover, it preserves the chosen absorbers $m \in M$ and $n \in N$, whence it is a morphism of $T$-groups. Conversely, if $f$ is a morphism of $T$-groups from $\big(\gG(M;m),\act_m\big)$ to $\big(\gG(N;n),\act_n\big)$, then $f = \cH(f)$ is a morphism of heaps of $T$-modules from $\cH\big(\hH\big(\gG(M;m)\big),\act_m\big) = M$ to $\cH\big(\hH\big(\gG(N;n)\big),\act_n\big) = N$ and it clearly satisfies $f(m) = n$.
\end{proof}

The following corollaries of Proposition \ref{prop:newTHmodTMod} highlight how heaps of modules behave as affine versions of $T$-groups. 

\begin{corollary}\label{cor:THmodTMod}
Let $T$ be a truss, $M,N$ be non-empty heaps of $T$-modules and let $f \colon  M \lra N$ be a function. Then 
\[f \in T\HMod(M,N) \quad  \iff \quad \tau_{f(m)}^{n} \circ f \in T\GMod\Big(\big(\gG(M;m),\act_m\big),\big(\gG(N;n),\act_n\big)\Big),\] 
for all $m\in M$ and $n\in N$. In particular, $f\colon M \lra N$ is a morphism of heaps of $T$-modules if and only if $f(-) = F(-) +_n f(m)$ for some morphism of $T$-groups $F\colon \big(\gG(M;m),\act_m\big) \lra \big(\gG(N;n),\act_n\big)$.
\end{corollary}

\begin{proof}
Recall from Lemma \ref{lem:trans} that $\tau_{f(m)}^n \in T\HMod(N,N)^{\times}$ and notice that $\big(\tau_{f(m)}^{n} \circ f\big)(m) = n$. Thus, by Proposition \ref{prop:newTHmodTMod}, $\tau_{f(m)}^{n} \circ f \in T\GMod\Big(\big(\gG(M;m),\act_m\big),\big(\gG(N;n),\act_n\big)\Big)$ if and only if $\tau_{f(m)}^{n} \circ f \in T\HMod(M,N)$, if and only if $f \in T\HMod(M,N)$.

To conclude, observe that by setting $F \coloneqq \tau^n_{f(m)} \circ f$ we find
\[
f(-) = \left(\tau_n^{f(m)} \circ \tau_{f(m)}^n \circ f\right)(-) = \left[F(-),n,f(m)\right] = F(-) +_n f(m). \qedhere
\]
\end{proof}

\begin{corollary}\label{cor:ringshoms}
Let $(G,+,0_G,\cdot)$ and $(H,+,0_H,\cdot)$ be $T$-groups, for some truss $T$. Then
\[
f \in T\HMod\big(\cH(\hH(G)), \cH(\hH(H))\big) \quad \iff \quad \tau_{f(0_G)}^{0_H} \circ f \in T\GMod\big(G,H\big).
\]
In particular, for $M,N$ two $T$-modules we have that for any $m\in M$ and $n\in N$,
\[
f\in T\HMod\big(\Hh(M),\Hh(N)\big) \iff \tau_{f(m)}^n\circ f \in T\GMod\big(\big(\gG(M;m),\triangleright_m\big),\big(\gG(N;n),\triangleright_n\big)\big).
\]
\end{corollary}

\begin{proof}
Recall from Lemma \ref{lem:homtoind} that $\big(\cH\big(\hH(G,+,0_G),\cdot\big),\cdot_{0_G}\big) = \big(\hH(G,+,0_G),\triangleright_{0_G}\big)$ and the latter is the $T$-group $(G,+,0_G,\cdot)$ again by Lemma \ref{lem:absorber} and Theorem \ref{thm:Tgroups}. Hence, the first claim follows from Corollary \ref{cor:THmodTMod}. The second claim is a particular instance of the first one.
\end{proof}

Now, by the following example we can observe that the analogue of Corollary \ref{cor:ringshoms} for $T$-modules does not hold, in general.

\begin{example}\label{ex:nocorr}
Let us consider a non-empty $T$-module $(M,\cdot)$ without an absorber. For instance, $M = \ZZ$ as a module over itself with truss structure given by $m \cdot n = 2mn + m + n$ (see \cite[Corollary 3.53]{Brz:par}). Consider also the induced submodule $(M,\triangleright_e)$, for an arbitrary element $e\in M$. 
Since the identity morphism and the constant map $\widehat{e}\colon M \lra M$, $m \lto e$, are elements of 
\[T\GMod\Big(\big((M,\triangleright_e),\triangleright_e\big),\big((M,\cdot),\triangleright_e\big)\Big) = T\GMod\big((M,\triangleright_e),(M,\triangleright_e)\big),\]
we know from Corollary \ref{cor:ringshoms} that the identity morphism and all the constant maps $\widehat{m}\colon M \lra M$, $n \lto m$, for $m \in M$ are elements of $T\HMod\big(\Hh(M,\triangleright_e),\Hh(M,\cdot)\big)$. However, $T\Mod\big((M,\triangleright_e),(M,\cdot)\big) = \varnothing$ and so no translation of any morphism in $T\HMod\big(\Hh(M,\triangleright_e),\Hh(M,\cdot)\big)$ can be therein.
\end{example}

Even though the $T$-modules $(M,\cdot)$ and $(M,\triangleright_e)$ are not necessarily isomorphic, unexpectedly, every induced submodule of $(M,\cdot)$ is an induced submodule of $(M,\triangleright_e)$ and {\it vice versa} (see the paragraphs preceding \eqref{eq:indchng}). Thus, we have an equality of heaps of $T$-modules $\Hh(M,\triangleright_e)=\Hh(M,\cdot)$. On the other hand, such an equality always yields an isomorphism of modules in case the truss is a ring or the modules are $T$-groups. This points out how unique and different modules over trusses are from $T$-groups and modules over rings.

\begin{proposition}
The category $\star/T\HMod$ is isomorphic to the category $T\GMod$.
\end{proposition}

\begin{proof}
In view of Corollary \ref{cor:THmodTMod}, the assignment sending every $\star \stackrel{f}{\lra} M$ in $\star/T\HMod$ to $\big(\gG(M;f(*)),\cdot_{f(*)}\big)$ gives rise to a well-defined, fully faithful functor. In the opposite direction, by Corollary \ref{cor:ringshoms}, the assignment sending every $(G,+,0,\cdot)$ in $T\GMod$ to $\star \lra \cH\big(\hH(G),\cdot\big): * \lto 0$, is a well-defined, fully faithful functor, which is the inverse of the previous one.
\end{proof}

Given a ring $R$, with every $R$-module $M$ one can associate a heap of $R$-modules by taking $\Lambda(r,m,m') = r\act m' - r\act m + m$ and $[m,m',m''] = m - m' + m''$, where $r\in R$ and $m,m',m''\in M$. Choosing these operations is the same as taking $\Hh(\tT(M))$ as in Example \ref{ex:modules}\eqref{ex:assomodule} and Proposition \ref{prop:homia}. By Proposition \ref{prop:homia} and Lemma \ref{lem:homtoind}, all heaps of $\tT(R)$-modules arise from $\tT(R)$-modules, but not necessarily from $R$-modules. Thus we are led to consider also the full subcategory $R\HMod$ of $\tT(R)\HMod$ in which all objects are coming from $R$-modules, that is, $\tT(R)$-modules with exactly one absorber (see \cite[Lemma~4.6~(2)(ii)]{BrzRyb:mod}).

\begin{proposition}\label{prop:RHMod}
The category $R\HMod$ is the category $\big(\tT(R)\HMod_{\mathbf{is}}^{\mathbf{cn}}\big)^{\mathbf{in}}$ of inhabited isotropic contractible heaps of $\tT(R)$-modules.
\end{proposition}

\begin{proof}
First of all, let us show that the category $R\HMod$ is closed under isomorphisms, whence it is the essential image of the composite functor
\[
\xymatrix{
R\Mod \ar[r]^-{\tT} & \tT(R)\Mod \ar[r]^-{\cH} & \tT(R)\HMod
}.
\]
Let $(M,[-,-,-],\Lambda)$ be a heap of $\tT(R)$-modules for which there exist an $R$-module $P$ and isomorphism of heaps of $\tT(R)$-modules  $\varphi \colon \cH(\tT(P)) \lra M$. Then $M$ is non-empty, as $e \coloneqq \varphi(0_P) \in M$, and $1_R \in \stab(M)$ and $0_R \in \Ann(M)$, 
so that it is isotropic and contractible. Thus, we may consider the abelian group $\gG(M;e) = (M, +_e, e)$ and the $R$-action $R \times M \lra M, (r,m) \lto r \act_e m$, which makes $M$ an $R$-module. 
In this setting, $\tT\big((M, +_e, e),\act_e\big) = \big((M,[-,-,-]),\act_e\big)$ and $\cH\big((M,[-,-,-]),\act_e\big) = (M,[-,-,-],\Lambda)$ by Lemma \ref{lem:homtoind}, which shows that $M$ is an object in $R\HMod$.

Now, if $M$ is an $R$-module, then $\cH(\tT(M))$ is an inhabited heap of $\tT(R)$-modules with $1_R \in \stab\big(\cH(\tT(M))\big)$ and $0_R \in \Ann\big(\cH(\tT(M))\big)$. Whence any object of $R\HMod$ is an object in $\big(\tT(R)\HMod_{\mathbf{is}}^{\mathbf{cn}}\big)^{\mathbf{in}}$, too. Conversely, for any object $(M,[-,-,-],\Lambda)$ in $\big(\tT(R)\HMod_{\mathbf{is}}^{\mathbf{cn}}\big)^{\mathbf{in}}$ we have that $M$ is non-empty, so there exists $e \in M$, and that $1_R \in \stab(M)$ and $0_R \in \Ann(M)$ by Theorem \ref{thm:IUmods}. Thus, as above we may consider the abelian group $\gG(M;e) = (M, +_e, e)$ and the $R$-action $R \times M \lra M, (r,m) \lto r \act_e m$, which makes of $M$ an $R$-module. 
Again, it turns out that $(M,[-,-,-],\Lambda) = \cH(\tT(\gG(M;e)))$ and hence it is an object in $R\HMod$.
\end{proof}

The following is a particular instance of Corollary \ref{cor:ringshoms} (see also \cite[Lemma 3.1]{BreBrz:Bae}).

\begin{proposition}\label{prop:morphRing}
Let $R$ be a ring and $M,N$ be $R$-modules. Then
\[
f\in R\HMod\big(\Hh(\tT(M)),\Hh(\tT(N))\big) \
\iff \
[f(-),f(0_M),0_N]\in R\Mod(M,N),
\]
where $0_M$ is the neutral element of the additive group $M$ and $0_N$ is the neutral element of $N$. In particular, $f\colon M \lra N$ is a morphism of heaps of $R$-modules if and only if $f = F + f(0_M)$ for some morphism of $R$-modules $F\colon M \lra N$.
\end{proposition}


\subsection{Heaps of modules as affine spaces}\label{ssec:aff}
In Section \ref{ssec:RHMod}, we have shown that every homomorphism of heaps of $T$-modules $f:\Hh(M)\lra\Hh(N)$ is a homomorphism of some particular choice of $T$-modules up to a translation $\tau$. Moreover, in the case of rings, or in more general setting of $T$-modules with absorbers, we can find a homomorphism of native $T$-modules $M$ and $N$ which up to a particular choice of translation $\tau$ is equal to $f$. In fact, this suggests a deep connection between heaps of modules and affine spaces. Thus our aim of this section is to reveal this connection. To motivate this discussion we adapt \cite[page 45]{Ostermann-Schmidt} to recall the following notion.

\begin{definition}\label{def:affinemods}
An \emph{affine module over a ring $R$} is a set $M$ together with maps $[-,-,-]\colon M \times M \times M \lra M$ and $\Lambda\colon  R \times M \times M \lra M$ such that
\begin{enumerate}[label=P\arabic*.,ref=P\arabic*,leftmargin=0.8cm]
\item\label{item:P1} $[b,a,c] = [c,a,b]$,
\item\label{item:P2} $[b,a,a] = b$,
\item\label{item:P3} $[[b,a,c],c,e] = [b,a,e]$,
\end{enumerate}
\begin{enumerate}[label=V\arabic*.,ref=V\arabic*,leftmargin=0.8cm]
\item\label{item:V0} $\big[\Lambda(r,a,b),a,c\big] = \Lambda\big(r,c,[b,a,c]\big)$
\item\label{item:V1} $\big[\Lambda(r,a,b),a,\Lambda(r,a,c)\big] = \Lambda\big(r,a,[b,a,c]\big)$,
\item\label{item:V2} $\big[\Lambda(r,a,b),a,\Lambda(s,a,b)\big] = \Lambda(r+s,a,b)$,
\item\label{item:V3} $\Lambda(rs,a,b) = \Lambda(r,a,\Lambda(s,a,b))$,
\item\label{item:V4} $\Lambda(1,a,b) = b$,
\end{enumerate}
for all $a,b,c,e \in M$ and $r,s\in R$.
\end{definition}

\begin{proposition}\label{prop:affinemods}
An affine $R$-module is exactly an isotropic and contractible heap of $\tT(R)$-modules.
\end{proposition}

\begin{proof}
Clearly, any abelian heap satisfies \ref{item:P1}--\ref{item:P3}. By Theorem \ref{thm:IUmods} \ref{item:IUmod2} and \ref{item:IUmod1} any isotropic contractible heap of $\tT(R)$-modules satisfies \ref{item:V4} and $\Lambda(0,a,b) = a$ for all $a,b\in M$.  
Thus, the only non-trivial check of \ref{item:V0}-\ref{item:V3} consists in observing that
\[
\begin{aligned}
\Lambda\big(r,c,[b,a,c]\big) & \stackrel{\eqref{eq:rightheap}}{=} \big[\Lambda(r,c,b),\Lambda(r,c,a),\Lambda(r,c,c)\big] \\
& \stackrel{\eqref{eq:idempotent}}{=} \big[\Lambda(r,c,b),\Lambda(r,c,a),a,a,c\big] \\
& \stackrel{\eqref{eq:basechange}}{=}\big[\Lambda(r,a,b),a,c\big].
\end{aligned}
\]
Thus an isotropic and contractible heap of $\tT(R)$-modules is an affine $R$-module.

Conversely, assume that $(M,[-,-,-],\Lambda)$ satisfies \ref{item:P1}--\ref{item:P3} and \ref{item:V0}--\ref{item:V4}. If $M$ is empty, then trivially $M$ is also a heap of $\tT(R)$-modules and $\stab(M)=R=\Ann(M)$. If $M$ is non-empty, then the fact that $(M,[-,-,-],\Lambda)$ is a heap of $\tT(R)$-modules such that $\Lambda(1,a,b) = b = \Lambda(0,b,a)$ for all $a,b \in M$ is a simple check. In fact, one may take advantage of the fact that for any chosen $a \in M$, all $b,c \in M$ and $r \in R$
\[b +_a c \coloneqq [b,a,c] \qquad \text{and} \qquad r\act_a b \coloneqq \Lambda(r,a,b)\]
 define an $R$-module structure on $M$ (see \cite[Satz 1]{Ostermann-Schmidt}). 
\end{proof}

As a consequence, it is natural to define morphisms of affine $R$-modules as morphisms of the corresponding heap of $\tT(R)$-module structure and Proposition \ref{prop:morphRing} confirms that they coincide with the intuitive idea of maps which are linear up to a constant.

\begin{corollary}\label{cor:summRmodules}
If $R$ is a ring, then the category of affine $R$-modules is (isomorphic to) the full subcategory $\tT(R)\HMod_{\mathbf{is}}^{\mathbf{cn}}$ of \label{p-contr-is-heaps} $\tT(R)\HMod$ consisting of isotropic contractible heaps of $\tT(R)$-modules. Moreover, the category of non-empty affine $R$-modules is (isomorphic to) the category $R\HMod$. 
\end{corollary}

\begin{proof}
It follows from Proposition \ref{prop:RHMod} and Proposition \ref{prop:affinemods}.
\end{proof}

\begin{remark}
By Proposition \ref{prop:morphRing}, all morphisms of inhabited isotropic contractible heaps of $\tT(R)$-modules (that is, of affine $R$-modules) are translations of $R$-module homomorphisms. In both cases of affine $R$-modules and heaps of modules the empty set is the initial object.
\end{remark}

Inspired by the fact that, by Proposition \ref{prop:newTHmodTMod}, morphisms of heaps of $T$-modules are essentially morphisms of $T$-groups up to fixing an "origin" and since considering affine $R$-modules amounts to considering (isotropic and contractible) heaps of $\tT(R)$-modules, let us consider the following extension of the well-known definition of an affine space as a free transitive action of the additive group of a vector space on a set (see \cite[Appendix \S2]{MacLane-Birk}).

\begin{definition}\label{def:Taffinespace}
Let $T$ be a truss. A {\it $T$-affine space} is a triple $(A,G,\varrho_A)$,
where $A$ is a set, $G$ is a $T$-group and $\varrho_A:G\xhookrightarrow{\varrho_A} \Sym A$ is a group monomorphism into the symmetric group of $A$, such that the shear map
\[G\times A\xrightarrow{(\varrho_A,\pi_2)} A\times A,\qquad (g,a)\lto (\varrho_A(g)(a),a) \]
is bijective.
\end{definition}

\begin{example}
Let $A$ be a non-empty set. Then $(A,G,\varrho_A)$ is a $T$-affine space
if and only if $G$ acts freely and transitively on $A$ by $\varrho_A$.
For the empty set there exists a unique $T$-affine space $(\varnothing,\{*\},\{*\}\hookrightarrow\{\id\})$.
\end{example}

\begin{remark}
Affine spaces over a truss $T$ are pseudo-torsors over $\{*\}$ in the sense of \cite[16.5.15]{Grothendieck}.
\end{remark}

\begin{example}
Recall that an affine space over a field $\FF$ is a triple $(A,V,\circ)$, where $A$ is a set, $V$ is a vector space over $\FF $ and $\circ\colon V\times A\lra A$ is a free and transitive action of the additive group of $V$ on $A$. If $G$ is a vector space over $\FF$ and $A$ is a non-empty set, then $(A,G,\circ)$ is an affine space if and only if $(A,G,\varrho_A)$ is $\tT(\FF)$-affine space with $\varrho_A(g)(a) = g\circ a$.
\end{example}

\begin{definition}\label{def:homos}
A {\it homomorphism of $T$-affine spaces} $(A,G,\varrho_A)$ and $(B,H,\varrho_B)$ is a pair $(F,f)\in\Set(A,B)\times T\GMod(G,H)$ such that $f$ is equivariant, that is, the diagram
\[
\xymatrix @R=20pt{
A\ar[rr]^{\varrho_A(g)}\ar[d]_-F && A \ar[d]^-F\\
A\ar[rr]^{\varrho_B(f(g))} &&B
}\]
commutes for all $g\in G$. The $T$-affine spaces and their homomorphisms form a category, which we denote by $\Aff_T$.\label{p-aff-t}
\end{definition}

\begin{remark}
The empty $T$-affine space in $\Aff_T$ is the initial object. This follows by the fact that $\{*\}$ and $\varnothing$ are initial objects in the respective categories.
\end{remark}

\begin{remark}\label{rem:affk}
Let us denote by $\Aff_{\FF}$ the category of affine spaces over a field $\FF$. \label{p-aff-F-space}
Recall that a homomorphism between two affine spaces $(A,V,\circ)$ and $(B,W,\bullet)$ is a map $F\colon A\lra B$ such that for all $a\in A$ and $v\in V$, $F(v\circ a)=f(v)\bullet F(a)$, where $f\colon V\lra W$ is a homomorphism of vector spaces. 
\note{As $V$ acts freely and transitively on $A$, for all $a,b\in A$ there exists unique $v\in V$ such that $v\circ a=b$; we denote it by $b - a$. Thus, $f(v) = f(b-a) = F(b)-F(a)$ for any pair $a,b\in A$ such that $v\circ a=b$.} 
Since every affine space over $\FF$ is a $\tT(\FF)$-affine space and $F(v\circ a)=f(v)\bullet F(a)$ is the equivariant condition in the sense of Definition \ref{def:homos}, we conclude that $(F,f)$ is a homomorphism of $\tT(\FF)$-affine spaces.
Hence $\Aff_\FF$ is a full subcategory of $\Aff_{\tT(\FF)}$.
\end{remark}

\begin{example}
 The truss $\star = \{*\}$ acts on any group $G$ by $*\cdot g=g$ as in Example \ref{ex:easy}\ref{item:easy2} and any group homomorphism is equivariant with respect to the unital action of $\star$. Thus the category of abelian groups is isomorphic to the category $\star\GMod_{\mathbf{is}}$.
\end{example}

The main aim of this section is to relate $T$-affine spaces to heaps of $T$-modules. We start by extending the translation group functor \eqref{trans.funct} to a functor between heaps of $T$-modules and $T$-groups.

\begin{lemma}\label{lem:Trans}
Let $(M,\Lambda)$ be a heap of $T$-modules. Then $\mathrm{Trans}(M)$ is a $T$-group with $T$-action
\begin{equation}\label{def.trans.act}
t\cdot \tau_a^b=\tau_a^{\Lambda(t,a,b)} \colon  m\lto [m,a,\Lambda(t,a,b)].
\end{equation}
Moreover, if $f\colon (M,\Lambda_M)\lra (N,\Lambda_N)$ is a morphism of heaps of $T$-modules, then the group homomorphism $\mathrm{Trans}(f)$ is $T$-linear. In particular, the functor $\mathrm{Trans}\colon \heap\lra \grp$ induces a functor $T\HMod\lra T\GMod$ which we denote by $\mathrm{Trans}$ again.
\end{lemma}

\begin{proof}
Let $(M,\Lambda_M)$ be a heap of $T$-modules. We already know from Section~\ref{ssec:heaps} that $\mathrm{Trans}(M)$ is an abelian group. The associativity of the action \eqref{def.trans.act} follows by the associative law  \eqref{eq:ass} for $\Lambda$. 
The distributivity of the heap operation of $T$ over the action \eqref{def.trans.act}, that is, the first of the properties \eqref{eq:TGrp}, follows by the fact that $\Lambda$ is a heap morphism  in the first argument \eqref{eq:leftheap} and by equation \eqref{comp.trans} combined with the commutativity of the group operation $\circ$ on $\mathrm{Trans}(M)$. 
Finally, 
\[
\begin{aligned}
\Lambda\big(t,a,[b,a',b']\big)
 & \stackrel{\eqref{eq:rightheap}}{=} \Big[\Lambda\big(t,a,b\big),\Lambda\big(t,a,a'\big),\Lambda\big(t,a,b'\big)\Big] \\
 & \stackrel{\phantom{\eqref{eq:basechange}}}{=} \Big[\Lambda\big(t,a,b\big),a',a',\Lambda\big(t,a,a'\big),\Lambda\big(t,a,b'\big)\Big] \\
 & \stackrel{\phantom{\eqref{eq:basechange}}}{=} \Big[\Lambda\big(t,a,b\big),a',\Lambda\big(t,a,b'\big),\Lambda\big(t,a,a'\big),a'\Big] \\
  &\stackrel{\eqref{eq:basechange}}{=} \Big[\Lambda(t,a,b),a',\Lambda(t,a',b')\Big],
\end{aligned}
\]
for all $t\in T$ and $a,b,a',b'\in M$. Combining this equality with \eqref{comp.trans} one obtains the distributivity of the $T$-action over the group operation, i.e.\ the second of properties \eqref{eq:TGrp}.

Concerning the morphisms, it is clear that if $f \colon  M \lra N$ is a morphism of heaps of $T$-modules, then
\[
\begin{aligned}
\mathrm{Trans}(f)\left(t\act \tau_a^b\right) & = \mathrm{Trans}(f)\big(\tau_a^{\Lambda(t,a,b)}\big) = \tau_{f(a)}^{f\big(\Lambda(t,a,b)\big)} 
 = \tau_{f(a)}^{\Lambda\big(t,f(a),f(b)\big)} = t \act \mathrm{Trans}(f)\left(\tau_a^b\right),
\end{aligned}
\]
and hence $\mathrm{Trans}(f)$ is morphism of $T$-groups.
\end{proof}

The following theorem gives a geometric interpretation of heaps of $T$-modules.

\begin{theorem}\label{thm:affT}
The categories $\Aff_{T}$ and $T\HMod$ are equivalent.
\end{theorem}

\begin{proof}
Let us begin by defining a functor from $\Aff_T$ to $T\HMod$. The empty $T$-affine space is sent to the unique empty heap of $T$-modules. Let $(A,G,\varrho_A)$ be a non-empty $T$-affine space and let us define the ternary operations $[-,-,-]$ and $\Lambda$ as follows. For all $a,b,c \in A$, there exists a unique $u \in (G,+)$ such that $\varrho_A(u)(b) = a$, and then we set
\begin{equation}\label{heap}
   [a,b,c] \coloneqq \varrho_{A}(u)(c). 
\end{equation}
Let us prove that this makes of $A$ a (non-empty) abelian heap. First of all, if $v \in G$ is the unique element such that $\varrho_A(v)(b) = c$, then
\[
\begin{aligned}
\varrho_A(u)(c) &= \varrho_A(u)(\varrho_A(v)(b))=(\varrho_A(u)\circ\varrho_A(v))(b) =\varrho_A(u+v)(b)\\ &=\varrho_A(v+u)(b)= \varrho_A(v)(\varrho_A(u)(b))= \varrho_A(v)(a),
\end{aligned}
\]
whence
\[[a,b,c] = \varrho_A(u)(c) = \varrho_A(v)(a) = [c,b,a].\]
Furthermore, it is clear that $[b,b,a] = \varrho_A(0)(a) = a$. Finally, let $a,b,c,d,e \in A$ and let $u,v \in (G,+)$ be such that $ \varrho_A(u)(b)=a$ and $\varrho_A(v)(d) = c$. In particular, \mbox{$\varrho_A(u)(c)=\varrho_A(u+v)(d)$.} Then we have
\[
[a,b,[c,d,e]] =\varrho_A(u)(\varrho_A(v)(e))=\varrho_A(u+v)(e)=[\varrho_A(u)(c),d,e]=[[a,b,c],d,e].
\]

Given $a,b \in A$, let $u \in G$ be the unique element such that $b = \varrho_A(u)(a)$. Then we can define
\[\Lambda(t,a,b) \coloneqq \varrho_A(t\act u)(a).\]
Let us check that $\Lambda$ makes  $(A,[-,-,-])$ a heap of $T$-modules. First,
$$
\begin{aligned}
\big[\Lambda(t,a,b),\Lambda(t',a,b),\Lambda(t'',a,b)\big] &\stackrel{\phantom{\eqref{eq:TGrp}}}{=} \big[\varrho_A(t\act u)(a),\varrho_A(t'\act u)(a),\varrho_A(t''\act u)(a)\big]\\ 
&\stackrel{\phantom{\eqref{eq:TGrp}}}{=} \varrho_A\big((t\act u)-(t'\act u)\big)\big(\varrho_A(t''\act u)(a)\big)\\
&\stackrel{\phantom{\eqref{eq:TGrp}}}{=} \varrho_A(t\act u - t'\act u + t''\act u)(a) \\
 &\stackrel{\eqref{eq:TGrp}}{=} \varrho_{A}\big([t,t',t'']\act u\big)(a) = \Lambda([t,t',t''],a,b).
\end{aligned}
$$
Second, for $a,b,c,d \in A$ let $u,v,w \in G$ be such that $b = \varrho_{A}(u)(a) $, $c = \varrho_A(v)(a)$, $d= \varrho_A(w)(a)$, then 
 $$
 [b,c,d] = \varrho_A(u-v)(d) = \varrho_A(u-v+w)(a),$$ 
 and so
\begin{multline*}
\big[\Lambda(t,a,b),\Lambda(t,a,c),\Lambda(t,a,d)\big] = \big[\varrho_A(t\act u)(a),\varrho_A(t\act v)(a),\varrho_A(t\act w)(a)\big] \\
 = \varrho_A\big(t\act u-t\act v\big)\big(\varrho_A(t\act w)(a)\big) \stackrel{\eqref{eq:TGrp}}{=} \varrho_A\big(t\act (u - v + w)\big)(a) = \Lambda(t,a,[b,c,d]).
\end{multline*}
Summing up, $\Lambda$ is a heap homomorphism in the first and third entry. Furthermore, with the same assumptions as above,
\[\Lambda(t,a,\Lambda(t',a,b)) = \Lambda\big(t,a,\varrho_A (t'\act u)(a)\big) = \varrho_A(t\act (t'\act u))(a)  =\varrho_A((tt')\act u) (a) = \Lambda(tt',a,b).\]
Finally, 
\begin{multline*}
\big[\Lambda(t,a,b),\Lambda(t,a,c),c\big] = \big[\varrho_A(t\act u)(a),\varrho_A(t\act v)(a) ,\varrho_A(v)(a)\big] = \varrho_A\big(t\act u-t\act v +v\big)(a) \\
 = \varrho_A\big(t\act (u-v)\big)\big(\varrho_A(v)(a)\big) =\varrho_A\big(t\act (u-v)\big)(c)= \Lambda(t,c,b).
\end{multline*}

Therefore, the assignment $(A,G,\varrho_A) \lto (A,[-,-,-],\Lambda)$ from $\Aff_T$ to $T\HMod$ is well-defined on objects. 

Now, recall that a morphism of affine spaces $(A,G,\varrho_A) \lra (B,G',\varrho_B)$ is given by a function $F\colon A \lra B$ and a $T$-group homomorphism $f\colon G \lra G'$ such that
\[F\big(\varrho_A(v)(a)\big) = \varrho_B(f(v))(F(a)) \]
for all $a \in A$, $v \in G$. Therefore, if $c = \varrho_A(v)(b)$ then $F(c) = \varrho_B(f(v))(F(b))$ and hence
\[\big[F(a),F(b),F(c)\big] = \varrho_B(f(v))(F(a)) = F\big(\varrho_A(v)(a)\big) = F\big([a,b,c]\big)\]
for all $a,b,c \in A$. Moreover, if $b = \varrho_A(u)(a)$ then $F(b) = \varrho_B(f(u))(F(a))$ and hence
\[
\begin{aligned}
\Lambda(t,F(a),F(b)) &= \varrho_B(t\act f(u))(F(a)) = \varrho_B(f(t\act u))(F(a))\\
&= F\big(\varrho_A(t\act u)(a)\big) = F\big(\Lambda(t,a,b)\big).
\end{aligned}
\]
Therefore, the assignment $(F,f)\lto F$ gives a well-defined function 
\[\Aff_T\big((A,G,\varrho_A),(B,G',\varrho_B)\big) \lra T\HMod\big((A,[-,-,-],\Lambda_A),(B,[-,-,-],\Lambda_B)\big).\]
Summing up, we have constructed a functor
\[\Phi\colon \Aff_T \lra T\HMod, \qquad \begin{cases} (A,G,\varrho_A) \lto (A,[-,-,-],\Lambda) \\ (F,f) \lto F \end{cases}.\]

In the opposite direction, let $(M,[-,-,-],\Lambda)$ be a heap of $T$-modules and let us consider the $T$-group $\mathrm{Trans}(M)$ with the action 
$\overline{\varrho_M} \colon \mathrm{Trans}(M) \lra \Sym M,$
given by the canonical inclusion. The shear map
\[\mathrm{Trans}(M) \times M \lra M \times M, \qquad \big(\tau_a^b,m\big) \lto \big(\tau_a^b(m),m\big),\]
has the inverse
$M \times M \lra \mathrm{Trans}(M) \times M$, $(m,n) \lto \big(\tau_n^m,n\big).$
Thus, $(M,\mathrm{Trans}(M),\overline{\varrho_M})$ is a $T$-affine space 

Moreover, if $F\colon  (M,\Lambda_M) \lra (N,\Lambda_N)$ 
is a morphism of heaps of $T$-modules, then the pair $(F,\mathrm{Trans}(F))$ is a morphism of $T$-affine spaces directly from the definition of $\mathrm{Trans}(F)$ and the actions $\overline{\varrho_M}$ and $\overline{\varrho_N}$.
Summing up, we have constructed a functor
\[\Psi\colon T\HMod \lra \Aff_T, \qquad \begin{cases} (M,[-,-,-],\Lambda_M) \lto (M,\mathrm{Trans}(M),\overline{\varrho_M}) \\ F \lto (F,\mathrm{Trans}(F)) \end{cases}.\]

To show that $\Psi$ is the quasi-inverse of $\Phi$, notice first of all that the heap operation and the action  on $\Phi\Psi(M,\Lambda)$ turn out to be
\[
\overline{\varrho_M}\big(\tau_a^b\big)(m) = [m,a,b] \quad \text{and}  \quad
\overline{\varrho_M}\big(t \act \tau_a^b\big)(a) = \overline{\varrho_M}\Big(\tau_a^{\Lambda(t,a,b)}\Big)(a) = \Lambda(t,a,b)
.
\]
Hence $\Phi\Psi$ is the identity on objects. 

On the other hand, 
$\Psi\Phi(A,G,\varrho_A)  = (A,\mathrm{Trans}(A),\overline{\varrho_A}).$   The definition \eqref{heap} of the heap operation on $A$ gives that $\tau_a^b = \varrho_A(g)$ for all $a,b\in M$, where $g$ is the unique element in $G$ such that $\varrho_A(g)(a)  = b$. Hence   $\mathrm{Trans}(A) = \{\varrho_A(g)\; |\; g\in G\}$.  In view of this and the arguments used to prove the associative law for the heap operation on $A$  the assignment
\[\epsilon_G\colon \mathrm{Trans}(A) \lra G, \qquad \tau_a^b \lto g, \quad \mbox{with} \quad \varrho_A(g)(a)  = b, \]
is an isomorphism of abelian groups. It commutes with $T$-actions since
for $g \in G$ such that $\varrho_A(g)(a)  = b$, 
    \[\epsilon_G(t\act \tau_a^b) = \epsilon_G\Big(\tau_a^{\Lambda(t,a,b)}\Big) = \epsilon_G(\tau_a^{\varrho_A(t\cdot g)(a)}) = t\act g = t \act \epsilon_G\big(\tau_a^b\big).\]
    
Therefore, we can consider the pair $(\mathrm{id}_A,\epsilon_G) \colon  (A,\mathrm{Trans}(A),\overline{\varrho_A}) \lra (A,G,\varrho_A)$, which is an isomorphism of $T$-affine spaces since
\[\mathrm{id}_A\Big(\overline{\varrho_A}\big(\tau_a^b\big)(m)\Big) = \varrho_A(g)(m) = \mathrm{id}_A(m) \circ \epsilon_G\big(\tau_a^b\big)\]
for all $a,b,m \in A$, where $g \in G$ is the unique element such that $\varrho_A(g)(a) = b$. We leave to the reader checking the naturality of $(\mathrm{id}_A,\epsilon_G)$.
\end{proof}

\begin{corollary}\label{cor:isotropic}
For a truss $T$, the categories $\Aff_{T}^{\mathbf{is}}$ and $T\HMod_{\mathbf{is}}$ are equivalent.
\end{corollary}

\begin{proof}
If $(A,G,\varrho_A)$ is an object in $\Aff_{T}^{\mathbf{is}}$ and if $u \in T$ is such that $u\cdot g = g$ for every $g \in G$, then for any $a,b\in A$ and $g\in G$ such that $\varrho_A(g)(a)=b$, we get that
$\Lambda(u,a,b)=\varrho_A(u\act g)(a)=\varrho_A(g)(a)=b$. Thus, $u \in \stab(A)$ and $(A,[-,-,-],\Lambda)$ is an isotropic heap of $T$-modules. Conversely, if $(M, [-,-,-],\Lambda)$ is an isotropic heap of $T$-module and $u \in \stab(M)$, then $\mathrm{Trans}(M)$ is an isotropic $T$-group as for all $a,b\in M$,
$
u \cdot \tau_a^b = \tau_a^{\Lambda(u,a,b)}=\tau_a^b.
$
Therefore $(M,\mathrm{Trans},\overline{\varrho_M})$ is an object in $\Aff_{T}^{\mathbf{is}}$. Moreover, since $\Aff_{T}^{\mathbf{is}}$ and $T\HMod_{\mathbf{is}}$ are full subcategories of $\Aff_{T}$ and $T\HMod$, respectively, the restrictions $\Phi|_{\Aff_{T}^{\mathbf{is}}}$ and $\Psi|_{T\HMod_{\mathbf{is}}}$ are well-defined, fully faithful, adjoint functors. Hence they form an equivalence. 
\end{proof}

Consider further the category $\Aff_{\star}^{\mathbf{is}}$ of affine spaces over the truss $T = \star$ whose objects are triples $(A,G,\varrho_A)$, where $G$ is an isotropic $\star$-group. \label{p-af-star-is}

\begin{corollary}\label{cor:classic}
The categories $\Aff_{\star}^{\mathbf{is}}$ and $\star\HMod_{\mathbf{is}}$ are equivalent.
\end{corollary}

\begin{remark}
Since $\star\HMod_{\mathbf{is}}$ is isomorphic to $\ahrd$ and the category of inhabited isotropic $\star$-affine spaces is isomorphic to the category of torsors over abelian groups, Corollary \ref{cor:classic} recovers the 
equivalence between inhabited abelian heaps and abelian torsors.
\end{remark}

\begin{corollary}\label{cor:afffield}
The category $\Aff_\FF$ of affine spaces over a field $\FF$ is equivalent to the category $\left(\tT(\FF)\HMod_{\mathbf{is}}^{\mathbf{cn}}\right)^{\mathbf{in}}$ of 
inhabited isotropic contractible heaps of $\tT(\FF)$-modules.
\end{corollary}

\begin{proof}
We keep the notation introduced in Remark \ref{rem:affk}.
By definition, an affine space $(A,V,\circ)$ over $\FF$ is a non-empty $\tT(\FF)$-affine space, whence the functor $\Phi$ sends it to the inhabited heap of modules $(A,[-,-,-],\Lambda)$ with
\[[a,b,c] = v \circ c \qquad \text{and} \qquad \Lambda(k,a,b) = kv \circ a,\]
where $v = b-a \in V$ is the unique element such that $v \circ a = b$. It is clear that
\[\Lambda(1,a,b) = v\circ a = b \qquad \text{and} \qquad \Lambda(0,a,b) = 0 \circ a = a\]
for all $a,b \in A$. Therefore, $\Phi$ restricts to
\[\Phi' \colon  \Aff_\FF \lra \left(\tT(\FF)\HMod_{\mathbf{is}}^{\mathbf{cn}}\right)^{\mathbf{in}}.\]
In the opposite direction, the functor $\Psi$ assigns every non-empty object $(M,[-,-,-],\Lambda)$ in $\tT(\FF)\HMod_{\mathbf{is}}^{\mathbf{cn}}$ to the $\tT(\FF)$-affine space $(M,\mathrm{Trans}(M),\overline{\varrho_M})$, where $\mathrm{Trans}(M)$ is a $\tT(\FF)$-group with action
\[k \cdot \tau_m^n = \tau_m^{\Lambda(k,m,n)}, \qquad \forall\,k \in \FF, m,n \in M,\]
and $\overline{\varrho_M}\colon \mathrm{Trans}(M) \lra \mathrm{Aut}(M)$ induces a free and transitive action
\[\mathrm{Trans}(M) \times M \lra M, \qquad \big(\tau_m^n,p\big) \lto [p,m,n].\]
In this case, and in view of Theorem \ref{thm:IUmods}, $\mathrm{Trans}(M)$ becomes a $\FF$-vector space because
$$
(k+k')\act\tau_a^b = \tau_a^{\Lambda\big([k,0,k'],a,b\big)}
= \tau_a^{[\Lambda(k,a,b),a,\Lambda(k',a,b)]} = \big(k\act \tau_a^b\big)\circ\big(k'\act\tau_a^b\big)
$$
and $1\cdot \tau_a^b = \tau_a^{\Lambda(1,a,b)} = \tau_a^b$.
Thus, $(M,\mathrm{Trans}(M),\circ)$ is an affine space over $\FF$ and $\Psi$ restricts to
\[\Psi'\colon \big(\tT(\FF)\HMod_{\mathbf{is}}^{\mathbf{cn}}\big)^{\mathbf{in}} \lra \Aff_\FF.\]
The pair of functors $(\Phi',\Psi')$ gives the stated equivalence.
\end{proof}

In this way, we may also recover the following fact.

\begin{corollary}
The category of affine spaces over a field $\FF$ described in terms of a binary operation $V \times A \lra A$ of a vector space on a set is equivalent to the category of affine spaces over $\FF$ described in terms of two ternary operations on a single non-empty set $A$ as in Definition \ref{def:affinemods}.
\end{corollary}

\begin{proof}
Both are equivalent to the category $\big(\tT(\FF)\HMod_{\mathbf{is}}^{\mathbf{cn}}\big)^{\mathbf{in}}$.
\end{proof}

\section{Examples and applications}\label{sec:appl}


We conclude by providing examples and applications of heaps of modules which are coming from the study of the set-theoretic solutions of the Yang-Baxter equation, non-commutative geometry and classical ring theory.

\subsection{The Baer-Kaplansky Theorem for \texorpdfstring{$T$}{T}-groups}\label{ssec:BK}

This first subsection is entirely devoted to an application of the theory developed in this paper, which represents the natural extension of the results from \cite{BreBrz:Bae} in view of what we proved in Section \ref{sec:aff}.

Let $S,T$ be trusses, $M$ be a $T$-group, $N$ be an $S$-group and let us denote by $\eE_T(M)$ the truss $T\HMod(\cH(M))$ of endomorphisms of the heap of $T$-modules $\cH(M)$.\label{page:ETM}

\begin{theorem}\label{thm:BK}
The trusses $\eE_T(M)$ and $\eE_S(N)$ are isomorphic if and only if there exists an isomorphism $\varphi\colon M \lra N$ of abelian groups and an isomorphism $\phi\colon T\GMod(M) \lra S\GMod(N)$ of trusses such that 
\begin{equation}\label{eq:varphilinear}
\varphi\big(F(m)\big) = \phi(F) \big(\varphi(m)\big)
\end{equation}
for all $m \in M$ and $F \in T\GMod(M)$.
\end{theorem}

\begin{proof}
The proof follows closely that of \cite[Theorem 3.3]{BreBrz:Bae}, whence we just sketch it and we leave the details to the interested reader.

Suppose that there exists an isomorphism $\varphi\colon M \lra N$ of abelian groups and an isomorphism $\phi\colon T\GMod(M) \lra S\GMod(N)$ of trusses such that  \eqref{eq:varphilinear} is satisfied. By Corollary \ref{cor:ringshoms}, $f \in \eE_T(M)$ if and only if $F \coloneqq \tau_{f(0)}^{0} \circ f \in T\GMod(M)$ and hence $f = F + f(0)$. Thus we may define
\[\Phi\colon \eE_T(M) \lra \eE_S(N), \qquad f \mapsto \phi(F) + \varphi(f(0)).\]
Since $\Phi(f)(0) = \varphi(f(0))$ because $\phi(F)(0) = 0$, it follows that  $\tau_{\Phi(f)(0)}^{0}\circ \Phi(f) = \phi(F)$ and hence $\Phi$ is well-defined by Corollary \ref{cor:ringshoms} again. The proof that it is an isomorphism of trusses is identical to the first part of the proof of \cite[Theorem 3.3]{BreBrz:Bae}.

Conversely, suppose that we have an isomorphism of trusses $\Phi\colon \eE_T(M) \lra \eE_S(N)$. By Example \ref{ex:constmaps} we know that the constant map $\widehat{m}$ is in $\eE_T(M)$ and one can show, as in the proof of \cite[Theorem 2.2]{BreBrz:Bae}, that for every $m \in M$, there exists a unique $\psi(m) \in N$ such that $\Phi(\widehat{m}) = \widehat{\psi(m)}$. This induces an isomorphism of abelian heaps $\psi\colon M \lra N$ which, by setting $\varphi \coloneqq \psi - \psi(0)$, induces an isomorphism $\varphi\colon M \lra N$ of abelian groups (see Corollary \ref{cor:AhAb}). Furthermore, by Corollary \ref{cor:ringshoms}, the assignment
\[\phi \colon T\GMod(M) \lra S\GMod(N), \qquad F \lto \tau_{\Phi(F)(0)}^0 \circ \Phi(F),\]
is a well-defined function. As in the second part of the proof of \cite[Theorem 3.3]{BreBrz:Bae}, one may show that $\phi$ is an isomorphism of trusses satisfying \eqref{eq:varphilinear},
ending the proof.
\end{proof}

\begin{remark}
The core of the argument for which Theorem \ref{thm:BK} holds may be expressed as follows. Any $T$-group $M$ is a module over its endomorphism truss $T\GMod(M)$. By \cite[Theorem~4.2]{BrzRyb:con}, we may consider the crossed product $M\overset{0}{\rtimes} T\GMod(M)$. Then Corollary \ref{cor:ringshoms} entails that the assignment
\[\eE_T(M) \lra M\overset{0}{\rtimes} T\GMod(M), \qquad f \lto \big(f(0),F\big)\]
is an isomorphism of trusses with inverse
\[M\overset{0}{\rtimes} T\GMod(M) \lra \eE_T(M), \qquad \big(m',F\big) \lto \big[m \mapsto F(m) + m'\big]. \qedhere\]
\end{remark}

We conclude this subsection with a somehow finer result that can be obtained in the commutative framework. Notice that, for $T$ a commutative truss and $M$ a $T$-group, the truss $\eE_T(M)$ gains a $T$-module structure given by $(t\cdot f)(m) = t \cdot f(m)$ for all $t \in T, f \in \eE_T(M)$ and $m \in M$, making of it a $T$-group (the distinguished absorber being $\widehat{0_M}$). In this setting, $\eE_T(M)$ can be seen as a heap of $T$-modules $\cH\big(\hH\big(\eE_T(M)\big)\big)$ as in Proposition \ref{prop:homia} and the binary composition law $\circ \colon \eE_T(M) \times \eE_T(M) \lra \eE_T(M)$ (granting the truss structure) is a morphism of heap of $T$-modules in both entries separately. That is to say, $\eE_T(M)$ is what we may call a \emph{heap of modules truss} over $T$.

\begin{theorem}
Let $T$ be a commutative truss and let $M$ and $N$ be $T$-groups. Then $M \cong N$ as $T$-groups if and only if $\eE_T(M) \cong \eE_T(N)$ as heap of modules trusses over $T$. Furthermore, for every isomorphism $\Psi \colon \eE_T(M) \lra \eE_T(N)$ of trusses which is also $T$-linear there exists a unique isomorphism $\psi\colon M \lra N$ of $T$-modules such that $\Psi(f) = \psi \circ f \circ \psi^{-1}$ for all $f \in \eE_T(M)$.
\end{theorem}

\begin{proof}
If $\varphi \colon M \lra N$ is an isomorphism of $T$-groups, then 
\[
\Phi \colon \eE_T(M) \lra \eE_T(N), \qquad f \lto \varphi \circ f \circ \varphi^{-1}, 
\]
is an isomorphism of trusses and of (heap of) $T$-modules as well, whence it is of heap of modules trusses over $T$ (observe that $\Phi$ is the morphism induced by $\varphi$ and $\phi\colon T\GMod(M) \lra T\GMod(N)$, $F \lto \varphi \circ F \circ \varphi^{-1}$, as in Theorem \ref{thm:BK}).
Conversely, if $\Phi \colon \eE_T(M) \lra \eE_T(N)$ is an isomorphism of heap of modules trusses over $T$, then 
the isomorphism $\varphi \colon M \lra N$ from the proof of Theorem \ref{thm:BK} is of $T$-groups.
Concerning the last statement, it follows by the same argument used to prove \cite[Theorem 2.2]{BreBrz:Bae}.
\end{proof}

\subsection{Shelves, spindles, racks, quandles and the Yang-Baxter equation}\label{ssec:spindle}
Heaps of modules give rise to examples of spindles and quandles that play an important role in the algebraic approach to knot theory and that also lead to solutions of the set-theoretic Yang-Baxter equation. We begin by recalling the necessary notions.

\begin{definition}[Shelves, spindles, racks and quandles]\label{def:quandle}
A {\em left shelf} is a set $X$ together with a left self-distributive binary operation $\diamond\colon  X\times X \lra X$, that is, for all $x,y,z\in X$,
\begin{equation}\label{self.distr}
x\diamond(y\diamond z) = (x\diamond y)\diamond(x\diamond z).
\end{equation}

A left shelf  $(X,\diamond)$ is called a {\em left spindle} provided that the operation $\diamond$ is idempotent, that is, for all $x\in X$, 
\begin{equation}\label{idem}
x\diamond x=x.
\end{equation}

A left shelf $(X,\diamond)$  (respectively spindle) is called a {\em left rack} (respectively {\em left quandle}) if $\diamond$ admits left division, that is, for all $x,y \in X$, there exists unique $x\bs y \in X$ such that
\begin{equation}\label{div}
x\diamond(x\bs y ) = y.
\end{equation}

A shelf, spindle, rack or quandle  $(X,\diamond)$ is said to be {\em entropic} (or {\em medial} or {\em abelian}), if, for all $x,y,z,w\in X$,
\begin{equation}\label{entropy.q}
(x\diamond y)\diamond(z\diamond w) = (x\diamond z)\diamond(y\diamond w).
\end{equation}

Morphisms of shelves, spindles, racks, quandles, are naturally defined. We denote by $\spin$ and $\quandle$ the categories of left spindles and quandles, respectively.\label{p-quandle}
\end{definition}

\begin{remark}\label{rem:entropy}
Note that the conditions \eqref{idem} and \eqref{entropy.q} imply the left self-distributivity \eqref{self.distr} as well as the right self-distributivity of $\diamond $. So, an entropic left spindle can be safely referred to simply as an entropic spindle. 
\end{remark}

\begin{remark}\label{rem:history}
Self-distributive operations that form a shelf have appeared already in logic in the works of C.S.\ Peirce \cite{Pei:alg}, but their possibly first systematic study is presented in \cite{BurMay:dis}. 

The notion of a quandle was introduced by D. Joyce \cite{Joy:cla} as an algebraic system that encodes the Reidemeister moves of knot theory.  In this context a spindle describes the first and the third Reidemeister moves. The term {\em rack} was coined by Fenn and Rourke in \cite{FenRou:rac}. They attribute this notion to J.C.\ Conway and G.C.\ Wraith who, in their unpublished correspondence, refer to it as a  {\em wrack}.
\end{remark}

\begin{example}\label{ex:affspin}
Let $(A,+)$ be an abelian group. For every endomorphism $f$ of $A$, define the operation $x\diamond _f y=x+f(y-x)$. Then $(A,\diamond _f)$ is an entropic spindle that will be called \textit{the affine spindle induced by $f$}. It enjoys the interesting property $x\diamond _f y=y\diamond _{\id-f} x$.

More generally, if $(H,[-,-,-])$ is an abelian heap and $f$ is an endomorphism of $H$, then $x\diamond _f y \coloneqq [f(y),f(x),x]$ defines a structure of entropic left spindle over $H$.
\end{example}

Recall from \cite{Dri:uns} that a function $r\colon  X\times X \lra X\times X$ is said to be a solution to the set-theoretic Yang-Baxter equation provided that the equality
\[
(r\times \id)\circ (\id \times r) \circ (r\times \id) = (\id\times r) \circ (r\times \id)\circ (\id \times r)
\]
holds in $X\times X\times X$. We do not ask for $r$ to be invertible, in general. We say that $r(x,y) = \big(r_1(x,y),r_2(x,y)\big)$ is non-degenerate \cite[Definition 1.1]{EtiSchSol-STS} if $r$ is invertible and 
\[x \lto r_2(x,y) \qquad \text{and} \qquad x \lto r_1(y,x)\]
are invertible as functions $X \lra X$ for any fixed $y \in X$.

Bearing in mind the role that the Yang-Baxter equation plays in the knot theory \cite{Tur:Yan} and that both self-distributivity and the Yang-Baxter equation correspond to the third Reidemeister move, the following lemma is not surprising.

\begin{lemma}[{\cite[\S9, Example 2]{Dri:uns}, see also \cite[Lemma 61]{Crans}}]\label{lem:rack.sol}
Let $X$ be a set with a binary operation $\diamond \colon  X \times X \lra X$. Then $(X,\diamond )$ is a left shelf if and only if the function
$$
r\colon  X\times X\lra X\times X, \qquad (x,y)\lto (x\diamond y,x),
$$
is a solution to the set-theoretic Yang-Baxter equation.
\end{lemma}

The key for unlocking the connection between heaps of modules and spindles and quandles can be found in the following result.

\begin{theorem}\label{thm:quandle}
Let $(M,\Lambda)$ be a heap of modules over a truss $T$ and let $u,v\in T$ be such that for some (equivalently, all) $m\in M$ and for all $n\in M$,
 $\Lambda (uv,m,n) =\Lambda(vu,m,n)$. Then, the operations 
\[
\diamond _u, \diamond _v\colon  M\times M \lra M, \qquad m\diamond _un \coloneqq \Lambda(u,m,n), \quad m\diamond _vn \coloneqq \Lambda(v,m,n),
\]
are idempotent and satisfy the entropic law, for all $x,y,z,w\in M$:
\[
(x\diamond _uy)\diamond _v(z\diamond _uw) = (x\diamond _vz)\diamond _u(y\diamond _vw).
\]

In particular, for any $u\in T$, $(M, \diamond _u)$ is an entropic {left} spindle and the assignment $(M,\Lambda) \lto (M,\diamond _u)$ induces a fully faithful functor $\Ff_u\colon T\HMod \lra \spin$. Consequently, for any $u\in T$, the function
$$
r_u\colon  M\times M\lra M\times M, \qquad (m,n)\lto (\Lambda(u,m,n),m),
$$
is a solution to the set-theoretic Yang-Baxter equation.

If, in addition, there exists $\bar{u}\in T$ such that, for all $m,n$,
\begin{equation}\label{div.lam}
\Lambda(u\bar{u}, m,n) = n,
\end{equation}
then $(M,\diamond _u)$ is an entropic left quandle and the functor $\Ff_u\colon T\HMod \lra \spin$ induces a fully faithful functor $\Ff_u\colon T\HMod \lra \quandle$. Consequently, for such an $u\in T$, the function
$$
r_u\colon  M\times M\lra M\times M, \qquad (m,n)\lto (\Lambda(u,m,n),m),
$$
is a non-degenerate solution to the set-theoretic Yang-Baxter equation with inverse
$$
r^{-1}_u\colon  M\times M\lra M\times M, \qquad (m,n)\lto (n,\Lambda(\bar{u},n,m)).
$$
\end{theorem}
\begin{proof}
The idempotent property of $\diamond _u$ follows from \eqref{eq:idempotent}, the entropic property \eqref{entropy.q} follows by Lemma~\ref{lem:entropy} and thus the self-distributivity holds by Remark~\ref{rem:entropy}. 

Furthermore, if $f\colon (M,\Lambda_M) \lra (N,\Lambda_N)$ is a morphism of heaps of $T$-modules, then
\[f\left(m \diamond _u n\right) = f\left(\Lambda_M(u,m,n)\right) \stackrel{\eqref{eq:morph}}{=} \Lambda_N(u,f(m),f(n)) = f(m) \diamond _u f(n)\]
for all $m,n \in M$, whence it is a morphism of left spindles.
The fact that $r_u$ is a solution of the set-theoretic Yang-Baxter equation follows by Lemma~\ref{lem:rack.sol}.

If there exists $\bar{u}$ such that \eqref{div.lam} holds, then, for all $m,n\in M$, we can define $m\bs n =\Lambda(\bar{u},m,n)$. The associativity of $\Lambda$ \eqref{eq:ass} together with \eqref{div.lam} imply \eqref{div}.
\end{proof}

\begin{example}\label{ex:HModaffspindle}
Let $(A,+)$ be an abelian group. Then the set of all affine spindles defined on $A$ as in Example \ref{ex:affspin} is induced by the structure of heap of $\tT(\mathrm{End}(A))$-modules on $A$ coming from Proposition \ref{prop:homia}, as in Theorem \ref{thm:quandle}. In fact, since $\Lambda(f,a,b) = \big[f(b),f(a),a\big]$ for all $a,b \in A$ and $f \in \mathrm{End}(A)$, we arrive exactly at $\Lambda(f,a,b) = a\diamond _f b$.
This allows us to organize the set of all the affine spindle structures on an abelian group $A$ into a heap of $\tT(\mathrm{End}(A))$-modules structure on $A$ itself.
It follows from Proposition \ref{prop:morphRing} that an endomorphism $\alpha$ of the induced abelian heap $(A,[-,-,-])$ is an endomorphism of this heap of $\tT(\mathrm{End}(A))$-modules structure if and only if there exists a central endomorphism $\overline{\alpha}$ of $(A,+)$ such that $\alpha(a)=\overline{\alpha}(a)+\alpha(0)$ for all $a\in A$.

A similar conclusion holds, more generally, for $H$ an abelian heap, considered as a heap of $\eE(H)$-modules as in Proposition \ref{prop:homia}.

Moreover, if $A$ is a left module over a ring $R$, the same construction induces a heap of $\mathrm{End}_R(A)$-modules. We observe that all multiplications $t_r \colon A\lra A$, $t_r(a)=r \cdot a$,, are morphisms of heap of $\mathrm{End}_R(A)$-modules.
\end{example}

\subsection{Connections, splittings and contractions}\label{ssec:split}

We now proceed to give some examples of heaps of modules which arise in a classical ring theory, starting with an example motivated by (non-commutative) geometry.

Let $(\Omega= \bigoplus_{n\in \NN} \Omega^n,d)$ be a (unital) differential graded algebra over a commutative ring $\k$. Denote by $\cO=\Omega^0$, the zero degree subalgebra of $\Omega$. Recall, for instance from \cite{Con:ncg}, that a {\em connection} on a left $\cO$-module $M$ is a $\k$-linear degree one map $\nabla\colon  \Omega\ot_\cO M \lra \Omega\ot_\cO M$ such that
\begin{equation}\label{connection}
\nabla(\omega \ot m) = d(\omega)\ot m + (-1)^k \omega \nabla(1\ot m),
\end{equation}
for all $\omega \in \Omega^k$ and $m\in M$. We denote the $k$-th component of $\nabla$ by $\nabla^k$. That is,
$$
\nabla^k\colon  \Omega^k\ot_\cO M \lra \Omega^{k+1}\ot_\cO M.
$$
Thanks to the connection Leibniz rule \eqref{connection}, $\nabla$ is fully determined by its zeroth component $\nabla^0$ (of which we can think as of a map $\nabla^0\colon  M\lra \Omega^1\ot_\cO M$).

In part dually, a {\em hom-connection} on a right $\cO$-module $M$ is defined as follows \cite{Brz:con}. Starting with a differential graded algebra  as before, we set $\Xi_k(M) = \lhom \cO {\Omega^{k}} M$, $k\in \NN$ and form the $\NN$-graded right  $\Omega$-module, 
$$
\Xi(M) =  \bigoplus_{k\in \NN} \Xi_k(M), \qquad  \big(\xi\cdot \omega\big)(\omega') = \xi(\omega\omega'), \quad \xi \in \Xi_{k+l}(M),\; \omega \in \Omega^k,\;\omega'\in \Omega^l.
$$
A hom-connection on $M$ is then a degree minus one $\k$-linear endomorphism $\Delta$ of $\Xi(M)$ such that, for all $\omega \in \Omega^{k}$ and $\xi \in \Xi_{k+1}(M)$,
 \[
 \Delta(\xi) (\omega) = \Delta(\xi\cdot \omega) + (-1)^{k+1} \xi(d\omega).
 \]
 The $k$-th component of $\Delta$, i.e., $\Delta\mid_{\Xi_{k+1}}$ is denoted by $\Delta_k$.

\begin{exlemma}[Heap of connections]\label{ex:connection}
Let $(\Omega= \bigoplus_{n\in \NN} \Omega^n,d)$ be a differential graded algebra over a commutative ring $\k$, $A$ be an algebra over $\k$, and $M$ be an $\cO$-$A$-bimodule. Consider  $\lend \k {\Omega\ot_\cO M}$  as a left $A$-module with the dual action
$
(a \act\gamma)(\omega \ot m) = \gamma(\omega\ot m\act a),
$
for all $a\in A$, $\gamma\in\lend{\k} {\Omega\ot_\cO M}$,  $\omega\in \Omega$ and $m\in M$.  

Let $\Gamma^k(M)\subseteq \lhom\k{\Omega^k\ot_\cO M}{\Omega^{k+1}\ot_\cO M}$ denote the $\k$-module of all $k$-components of connections on $M$. Then $\Gamma^k(M)$ is an induced $\tT(A)$-submodule of $\lend \k {\Omega\ot_\cO M}$, and hence a heap of $\tT(A)$-modules with the action
\begin{equation}\label{lam.conn}
\Lambda(a,\nabla^k_1,\nabla^k_2) = \nabla^k_1- a\act \nabla^k_1 +a\act \nabla^k_2.
\end{equation}

Furthermore, for all $u\in A$, $\Gamma^k(M)$ is a left spindle with the operation 
$$
\nabla_1^k\diamond _u\nabla_2^k = \nabla_1^k -u\act \nabla_1^k + u\act \nabla_2^k.
$$
Consequently, the function
$$
r_u\colon  \Gamma^k{(M)}\times \Gamma^k{(M)}\lra \Gamma^k{(M)}\times \Gamma^k(M), \quad (\nabla_1^k,\nabla_2^k)\lto  (\nabla_1^k -u\act\nabla_1^k + u\act\nabla_2^k,\nabla_1^k),
$$
is  a solution to the set-theoretic Yang-Baxter equation.

If $u$ is a unit in $A$, then $(\Gamma^k(M),\diamond _u)$ is an entropic quandle.
\end{exlemma}

\begin{proof}
It is clear that if $\nabla_1^k$, $\nabla_2^k$, $\nabla_3^k$ are the $k$-th components of connections $\nabla_1,\nabla_2,\nabla_3$ on $M$, respectively, then so does their heap bracket $[\nabla_1^k, \nabla_2^k, \nabla_3^k] = \nabla_1^k- \nabla_2^k +\nabla_3^k$. Thus we only need to check if the map $\Lambda(a,\nabla^k_1,\nabla^k_2)$ in \eqref{lam.conn} is the $k$-th component of a connection on $M$. Set $\Lambda(a,\nabla_1,\nabla_2) \coloneqq \nabla_1 - a\cdot \nabla_1 + a \cdot \nabla_2$. For all $\omega\in \Omega^k$ and $m\in M$ we observe that
$$
\begin{aligned}
\Lambda(a,\nabla_1,\nabla_2) &(\omega \ot m) = \nabla_1(\omega \ot m)- \nabla_1(\omega \ot m
\act a) +\nabla_2(\omega \ot m\act a)\\
&= d(\omega)\ot m + (-1)^k \omega \nabla_1(1\ot m) - d(\omega)\ot m\act a\\
& \quad + (-1)^{k+1} \omega \nabla_1(1\ot m\act a) + d(\omega)\ot m\act a 
+ (-1)^{k} \omega \nabla_2(1\ot m\act a)\\
&= d(\omega)\ot m + (-1)^k \omega \Big(\nabla_1(1\ot m)
-  \nabla_1(1\ot m\act a) 
+  \nabla_2(1\ot m\act a)\Big)\\
&= d(\omega)\ot m + (-1)^k\omega \Lambda(a,\nabla_1,\nabla_2)(1 \otimes m).
\end{aligned}
$$
It follows then that $\Lambda(a,\nabla_1,\nabla_2)$ is still a connection on $M$ whose $k$-th component is exactly $\Lambda(a,\nabla^k_1,\nabla^k_2)$, as required.

The second part of example follows by Theorem~\ref{thm:quandle}.
\end{proof}

\begin{remark}
For the sake of notation, write $\omega \in \Omega$ as $\sum_{k}\omega_k$, with $\omega_k \in \Omega^k$.
Notice that $f \in \lend \k {\Omega\ot_\cO M}$ is the $k$-component of a connection on $M$ if and only if there exists a $\k$-linear map $f^0\colon  M \lra \Omega^1\ot_\cO M$ such that
\[f(\omega \ot_\cO m) = d(\omega_k) \ot_\cO m + (-1)^k\omega_k f^0(m).\]
Since $\Gamma^k(M)$ is an induced $\tT(A)$-submodule of $\lend \k {\Omega\ot_\cO M}$, there exists a congruence for which $\Gamma^k(M)$ is a congruence class. The submodule of $\lend \k {\Omega\ot_\cO M}$ for which $\Gamma^k(M)$ is a congruence class is the submodule 
$$S(k) \coloneqq \left\{f\in\lend \k {\Omega\ot_\cO M} ~\left\vert~ \begin{gathered} \exists\, f^0\in \Hom_\k(M,\Omega^1 \ot_\cO M) \\ \textrm{and } f(\omega\otimes_\cO m)=(-1)^k\omega_k f^0(m), \\ \forall \, \omega\in\Omega \textrm{ and } m\in M  \end{gathered} \right.\right\}.$$ 
Also the $\k$-submodule $\Gamma(M) \subseteq \lend \k {\Omega\ot_\cO M}$ of all the connections on $M$ is a heap of $T(A)$-modules with the same structure and it is a congruence class with respect to the submodule
\[
S \coloneqq \left\{f\in\lend \k {\Omega\ot_\cO M} ~\left\vert~ \begin{gathered} f(\omega\otimes m) = {\textstyle \sum_k }(-1)^k\omega_k f(1\otimes m) \\  \forall\,  \omega\in\Omega \textrm{ and }  m\in M \end{gathered} \right.\right\}. \qedhere
\]
\end{remark}

\begin{example}[Truss of connections]
As in Example~\ref{ex:connection},
let $(\Omega= \bigoplus_{n\in \NN} \Omega^n,d)$ be a differential graded algebra over a commutative ring $\k$, $A$ be an algebra of $\k$, $M$ be an $\cO$-$A$-bimodule that is finitely generated and projective as a left $\cO$-module. It is well known (see e.g.\ \cite[Section~8]{CunQui:alg}) that the set of connections $\Gamma^1(M)$ is not empty. In particular, for any finite dual basis $e_i\in M$,  $\eps^i\in \lhom \cO M \cO$, there is the associated Grassmannian connection,
\[
\nabla\colon  M\lra \Omega^1\otimes_\cO M, \qquad m \lto \sum_i d(\eps^i(m)) \ot e_i.
\]
Using the heap of $A$-modules structure of $\Gamma^1(M)$ described in Example~\ref{ex:connection} we can construct the cross product truss $\cT(\nabla)\coloneqq \Gamma^1(M)\overset{\nabla}{\rtimes} \tT(A)$ as in Proposition~\ref{prop:cross} with the product:
$$
(\nabla_1,a)(\nabla_2,b) = \left(\nabla_1+ a\cdot \Big(\nabla_2 - \sum_i d\circ \eps^i\ot e_i\Big), ab\right).
$$
In view of Lemma~\ref{lem:cross}, $\Gamma^1(M)$ is a left $\cT(\nabla)$-module with the action
$$
(\nabla_1,a)\cdot \nabla_2 = \nabla_1+ a\cdot (\nabla_2 - \sum_i d\circ \eps^i\ot e_i),
$$
that is, for all $m\in M$,
$$
\Big((\nabla_1,a)\cdot \nabla_2\Big)(m) = \nabla_1(m)+  \nabla_2(m\act a) - \sum_i d(\eps^i(m\act a))\ot e_i.
$$
\end{example}

\begin{example}[Heap of hom-connections]\label{ex:hom-connection}
Let $(\Omega= \bigoplus_{n\in \NN} \Omega^n,d)$ be a differential graded algebra over  a commutative ring $\FF$, $A$ be an algebra of $\k$, $M$ be an $A$-$\cO$-bimodule. Using the left $A$-action on $M$ we consider  $\lend \k {\Xi( M)}$  as a left $A$-module in the standard way.
Then, 
the $\k$-module of all $k$-components of hom-connections on $M$, is an induced $\tT(A)$-submodule of $\lend \k {\Xi(M)}$, and hence a heap of $\tT(A)$-modules with the action
\[
\Lambda(a,\Delta_k,\Delta'_k) = \Delta_k - a\act \Delta_k +a\act \Delta'_k.
\]
By Theorem~\ref{thm:quandle},  $\Gamma_k(M)$  is a left spindle with the operation  
$\Delta_k\diamond _u\Delta'_k = \Delta^k -u\act \Delta_k + u\act \Delta'_k$, for all $u\in A$ (an entropic quandle if $u$ is a unit in $A$). Consequently, the function
$$
r_u\colon  \Gamma_k{(M)}\times \Gamma_k{(M)}\lra \Gamma_k{(M)}\times \Gamma_k(M), \quad (\Delta_k,\Delta'_k)\lto  (\Delta_k -u\act \Delta_k + u\act \Delta'_k,\Delta_k),
$$
is  a solution to the set-theoretic Yang-Baxter equation.
\end{example}

\begin{example}[Heap of splittings of short exact sequences]\label{ex:split}
Let $B$ be a $\k$-algebra over a commutative ring $\k$ and let $P,Q,R$ be left $B$-modules. Assume that
\[
 0 \lra P \overset{\iota}\lra Q \overset{\pi}\lra R \lra 0,
\]
is a short exact sequence of left $B$-modules, that is to say, an extension of $R$ by $P$ in $B$-$\textbf{Mod}$. If $R$ is a $B$-$A$ bimodule, for an (unital) algebra $A$, then the $\k$-module $\Hom_B(R,Q)$ is a  left $A$-module via
\[(a\cdot \phi)(x) = \phi(x\cdot a). \]
Consider the $\k$-submodule  $\Sigma(\pi) \coloneqq \{\sigma\colon R \lra Q \mid \pi \circ \sigma = \id_R\} \subseteq \Hom_B(R,Q)$. It is a sub-heap of the heap $\mathrm{H}\big(\Hom_B(R,Q)\big)$ because
\[\pi \circ [\rho,\sigma,\tau] = \pi\circ \rho - \pi \circ \sigma + \pi \circ \tau = \id_R,\]
and it becomes a heap of $\tT(A)$-modules over the truss $\tT(A)$ associated with $A$ with the ternary operation defined as follows
\[
\Lambda\colon  \tT(A) \times \Sigma(\pi) \times \Sigma(\pi) \lra \Sigma(\pi), \qquad (a,\tau,\sigma) \lto a\cdot \sigma + (1 -a )\cdot \tau \eqqcolon a \triangleright_\tau \sigma.
\]

Furthermore, by Theorem~\ref{thm:quandle}, for all $u\in A$, $\Sigma(\pi)$ is a spindle with the operation
$$
\tau \diamond _u\sigma = u\cdot \sigma + (1 -u )\cdot \tau, \qquad r\lto \sigma(ru) +\tau(r(1-u)).
$$
Consequently, the function
$$
r_u\colon  \Sigma(\pi)\times \Sigma(\pi)\lra \Sigma(\pi)\times \Sigma(\pi), \qquad (\tau,\sigma) \lto (u\cdot \sigma + (1 -u )\cdot \tau, \tau),
$$
is  a solution to the set-theoretic Yang-Baxter equation.

Whenever $u$ is a unit in $A$, $(\Sigma(\pi), \diamond _u)$ is an entropic quandle.

\noindent Observe that $\Sigma(\pi)$ is not a $T(A)$-submodule since, in general,
\[\pi\circ \left(a\cdot \sigma\right)(r) = r \cdot a \neq r.\]

\noindent In this case the submodule for which $\Sigma(\pi)$ is a congruence class is the submodule  
$$
S(\pi)\colon =\{f\in \Hom_A(R,Q)\ |\ \im(f)\subseteq\ker(\pi)\}
$$
of $\Hom_A(R,Q)$.

There is some overlap of this example with Example~\ref{ex:connection}. Recall for instance from \cite{Con:ncg} that a {\em universal differential envelope} of an algebra $B$ is a tensor algebra of the $B$-bimodule $\Omega^1 \coloneqq \ker m_B$, where $m_B\colon B\ot_\k B \lra B$ is the multiplication map of $B$, with the differential given on $B$ by $d(b) = 1\ot b - b\ot 1$, and extended to the whole of $\Omega$ by the graded Leibniz rule. It has been observed in \cite{CunQui:alg} that the degree-zero components of connections on a left $B$-module $M$ with respect to the universal envelope $\Omega$ are in bijective correspondence with the splittings of the short exact sequence
$$
\xymatrix{
0 \ar[r] & \Omega^1\ot_B M\ar[r]^-\iota & B\ot_\k M \ar[r]^-\pi & M \ar[r] & 0,}
$$
where
$$
\iota \colon  \left(\sum_ib_i\ot b'_i\right) \ot m \lto \sum_i b_i\ot b'_i\act m, \qquad \pi\colon  b\ot m \lto b\act m.
$$
Thus $\Gamma^0(M) = \Sigma(\pi)$.

\end{example}

\begin{example}[Truss of splittings]
In the set-up of Example~\ref{ex:split} consider a split exact sequence of left $B$-modules
$$
\xymatrix{0 \ar[r] & P\ar[r]^\iota & Q\ar@<0.5ex>[r]^\pi &R \ar@<0.5ex>[l]^{\varrho} \ar[r]& 0}.
$$
If $R$ is a $B$-$A$-bimodule, then denote by $\Sigma(\pi)$ the heap of $A$-modules of all splittings of $\pi$. Then, by Proposition~\ref{prop:cross}, $\Sigma(\pi)\overset{\varrho}{\rtimes} \tT(A)$ is a truss on the heap $\Sigma(\pi)\times \tT(A)$ with the product
$$
(\sigma_1,a)(\sigma_2,b) = (\sigma_1+a\cdot(\sigma_2 -\varrho), ab).
$$
\end{example}

\begin{example}[Heap of retractions]\label{ex:retraction}
Let $B$ be a $\k$-algebra  and let 
$$
 0 \lra P \overset{\iota}\lra Q \overset{\pi}\lra R \lra 0,
 $$
be a short exact sequence of right $B$-modules.
If $P$ is an $A$-$B$-bimodule, for an (unital) algebra $A$, then the $\k$-module $\Hom_B(Q,P)$ is a left $A$-module via
$(a\cdot \phi)(x) = a\cdot\phi(x).$
Consider the $\k$-submodule  
$$
\Rr(\iota) \coloneqq \{\rho\colon Q \lra P \in \Hom_B(Q,P) \mid \rho \circ \iota = \id_P\} \subseteq \Hom_B(Q,P).
$$
It is a sub-heap of the heap $\mathrm{H}\big(\Hom_B(Q,P)\big)$ because
\[[\rho,\sigma,\tau] \circ \iota = \rho \circ \iota - \sigma \circ \iota + \tau\circ \iota = \id_P,\]
and it becomes a heap of $\tT(A)$-modules through
$$
\Lambda\colon  \tT(A) \times \Rr(\iota) \times \Rr(\iota) \lra \Rr(\iota), \qquad (a,\tau,\sigma) \lto a\cdot \sigma + (1 -a )\cdot \tau \eqqcolon a \triangleright_\tau \sigma.
$$

There is also an overlap between this and Example~\ref{ex:hom-connection}. As shown in \cite[Section~3.9]{Brz:con} and \cite[Theorem~2.2]{BrzElK:int}, in the case of the universal differential graded algebra, there is a bijective correspondence between hom-connections on a right $B$-module $M$ and  retractions of the map $\iota \colon  M\lra \lhom \k {B} M$, given by $\iota(m)(b) = m\act b$. 

\end{example}

\begin{example}[Heap of chain contractions]
Let $A,B$ be two $\k$-algebras. Let $(C_\bullet,c_\bullet)$  be a chain complex of $B$-$A$-bimodules and $(D_\bullet,d_\bullet)$ be a chain complex of $B$-modules. Let also $f_\bullet \colon  C_\bullet \lra D_\bullet$ be a chain map of the underlying chain complexes of left $B$-modules, that is, for every $n \in \ZZ$ the function $f_n\colon  C_n \lra D_n$ is a morphism of left $B$-modules and all the squares in the following diagram of left $B$-modules  are commutative
\[
\xymatrix{
\cdots \ar[r] & C_{n+1} \ar[r]^-{c_{n+1}} \ar[d]_-{f_{n+1}} & C_n \ar[r]^-{c_n} \ar[d]^-{f_n} & C_{n-1} \ar[r] \ar[d]^-{f_{n-1}} & \cdots \\
\cdots \ar[r] & D_{n+1} \ar[r]_-{d_{n+1}} & D_n \ar[r]_-{d_n} & D_{n-1} \ar[r] & \cdots
}
\]
Denote by $\Sigma(f)$ the set of all chain contractions of $f$ (see \cite[Definition 1.4.3]{Weibel}), 
\[\Sigma(f) = \Big\{\{s_n\colon C_n \lra D_{n+1} \textrm{ of }B\textrm{-modules}\mid n\in\ZZ\} ~\Big\vert~ d_{n+1} \circ s_n + s_{n-1}\circ c_n = f_n\Big\}.\]
It is a sub-heap of the heap $\mathrm{H}\left(\Hom_B^{\mathrm{gr}}\left(\bigoplus_{n}C_n,\bigoplus_nD_{n+1}\right)\right)$ 
of degree-one morphisms of left $B$-modules,
because (dropping the subscripts for the sake of clarity)
\[d \circ [s,s',s''] + [s,s',s'']\circ c = (d\circ s + s \circ c) - (d \circ s' + s'\circ c) + (d \circ s'' + s'' \circ c) = f.\]
In addition, $\Sigma(f)$ is a heap of $\mathrm{T}(A)$-modules via 
\[\Lambda\colon \tT(A) \times \Sigma(f) \times \Sigma(f) \lra \Sigma(f), \quad (a,s,s') \lto a\cdot s' + (1-a) \cdot s.\]
In fact, for every $s\in \Sigma(f)$ and $a \in A$, the function $a\cdot s$ defined by
\[(a\cdot s)_n\colon C_n \lra D_{n+1}, \qquad x \lto s_n(x \cdot a),\]
is still left $B$-linear and it satisfies
\begin{multline*}
\left(d_{n+1} \circ (a \cdot s)_n + (a \cdot s)_{n-1}\circ c_n\right)(x) = d_{n+1}\big(s_{n}(x \cdot a)\big) + s_{n+1}\big(c_n(x)\cdot a\big) \\
= d_{n+1}\big(s_{n}(x \cdot a)\big) + s_{n+1}\big(c_n(x\cdot a)\big) = f_n(x \cdot a) = (a\cdot f_n)(x)
\end{multline*}
for all $x \in C_n$ and all $n \in \ZZ$. Therefore,
\begin{multline*}
d \circ \left(a\cdot s' + (1-a) \cdot s\right) + \left(a\cdot s' + (1-a) \cdot s\right) \circ c \\
= \big(d\circ (a \cdot s') + (a \cdot s')\circ c\big) + \big(d\circ ((1-a) \cdot s) + ((1-a) \cdot s)\circ c\big) \\
 = a \cdot f + (1-a) \cdot f = f.
\end{multline*}
In the same way, one may prove that if $(C_\bullet,c_\bullet)$ is a chain complex of right $A$-modules, $(D_\bullet,d_\bullet)$ is a chain complex of $B$-$A$-bimodules and $f_\bullet \colon  C_\bullet \lra D_\bullet$ is a chain map of the underlying chain complexes of right $A$-modules, then the set $\Sigma(f)$ of all chain contractions of $f$ is still a sub-heap of the heap $\mathrm{H}\big(\Hom_A^{\mathrm{gr}}(\bigoplus_{n}C_n,\bigoplus_nD_{n+1})\big)$ and a heap of $\mathrm{T}(B)$-modules with respect to the ternary operation
\[\Lambda\colon \tT(B) \times \Sigma(f) \times \Sigma(f) \lra \Sigma(f), \quad (b,s,s') \lto b\cdot s' + (1-b) \cdot s.\]
Notice that Examples \ref{ex:split} and \ref{ex:retraction} are particular cases of the present example. Regarding Example \ref{ex:split}, if we consider the chain morphism
\[\xymatrix{
\cdots \ar[r] & 0 \ar[r] & 0 \ar[r] \ar[d] & 0 \ar[r] \ar[d] & R \ar[r] \ar@{=}[d] \ar@{.>}[dl]^-{\sigma} & 0 \ar[r] & \cdots \\
\cdots \ar[r] & 0 \ar[r] & P \ar[r]^-{\iota} & Q \ar[r]^-{\pi} & R \ar[r] & 0 \ar[r] & \cdots,
}\]
then chain contractions $(\ldots,0,\sigma,0,\ldots)$ are in bijective correspondence with splittings $\sigma$ of the short exact sequence of left $B$-modules $0 \lra P \xrightarrow{\iota} Q \xrightarrow{\pi} R \lra 0$. Concerning Example \ref{ex:retraction} instead, if we consider the chain morphism (now in the category of right $B$-modules)
\[\xymatrix{
\cdots \ar[r] & 0 \ar[r] & P \ar[r]^-{\iota} \ar@{=}[d] & Q \ar[r]^-{\pi} \ar[d] \ar@{.>}[dl]^-{\rho} & R \ar[r] \ar[d] & 0 \ar[r] & \cdots \\
\cdots \ar[r] & 0 \ar[r] & P \ar[r] & 0 \ar[r] & 0 \ar[r] & 0 \ar[r] & \cdots ,
}\]
then chain contractions $(\ldots,0,\rho,0,\ldots)$ are in bijective correspondence with retractions $\rho$ of $\iota$ in the short exact sequence of right $B$-modules $0 \lra P \xrightarrow{\iota} Q \xrightarrow{\pi} R \lra 0$.
\end{example}

\begin{exlemma}[Heap of derivations]
Let $T$ be a truss. A {\em derivation} of $T$ is a heap endomorphism $D:T\lra T$, such that, for all $s,t\in T$,
\begin{equation}\label{Leibniz}
  D(st) = [D(s)t, st, sD(t)].
\end{equation}
As observed in \cite{Brz:Lie} the set of all derivations of $T$, $\mathrm{Der}(T)$, is a sub-heap of $\eE(T)$. If $T$ is commutative, then $\mathrm{Der}(T)$ is a heap of $T$-modules with the action
$$
\Lambda(t,D_1,D_2) = [D_1,t\act D_1,t\act D_2] : s\lto [D_1(s),tD_1(s),tD_2(s)],
$$
for all $D_1,D_2 \in \mathrm{Der}(T)$ and $s,t\in T$.
\end{exlemma}

\begin{proof}
Since $\Lambda$ is the restriction of the heap of modules action of $T$ on its endomorphism truss $\eE(T)$ as in Example \ref{ex:E(T)}, we only need to check that $\Lambda(t,D_1,D_2)$ is a derivation of $T$. Since both $D_1$ and $D_2$ satisfy \eqref{Leibniz} and $T$ is a commutative truss, for all $s,s'\in T$,
$$
\begin{aligned}
\Lambda(t,D_1,D_2)(ss') &=  [D_1(ss'),tD_1(ss'),tD_2(ss')]\\
&=[D_1(s)s',ss', sD_1(s'), tD_1(s)s',tss', tsD_1(s'), tD_2(s)s',tss', tsD_2(s')]\\
&=[D_1(s)s',tD_1(s)s', tD_2(s)s', ss', sD_1(s'),  stD_1(s'),  stD_2(s')]\\
&= [\Lambda(t,D_1,D_2)(s)s',ss',s\Lambda(t,D_1,D_2)(s')],
\end{aligned}
$$
where, in addition, we have also used the distributive laws of trusses (to derive the second and fourth equalities), and rules of reshuffling and cancellation stemming from the definition of an abelian heap. Hence the operation $\Lambda$ is well defined on $\mathrm{Der}(T)$.
\end{proof}

\section*{Acknowledgements}
The research of S.\ Breaz is supported by a grant of the Ministry of Research, Innovation and Digitization, CNCS/CCCDI--UEFISCDI, project number PN-III-P4-ID-PCE-2020-0454, within PNCDI III. 

The research  of T.\ Brzezi\'nski is partially supported by the National Science Centre, Poland, grant no. 2019/35/B/ST1/01115. 

P.\ Saracco is a Charg\'e de Recherches of the Fonds de la Recherche Scientifique - FNRS and a member of the ``National Group for Algebraic and Geometric Structures and their Applications'' (GNSAGA-INdAM).

The research of B.\ Rybo\l owicz is supported by the EPSRC grant EP/V008129/1.

We would like to thank Agata Pilitowska for historical (and terminological) information about shelves, spindles, racks and quandles which led to Remark~\ref{rem:history}.

\clearpage

\begin{appendix}

\section{Index of frequently used notation} \label{appendix}

{
\small

\begin{multicols}{2}

\subsection{Categories}

\begin{description}[itemsep=1.5pt,leftmargin=19pt]

\item[$\Ab$] abelian groups; \pageref{p-cat-ab}

\item[$\Aff_\FF$] affine spaces over a field $\FF$; \pageref{p-aff-F-space}

\item[$\Aff_{\star}^{\mathbf{is}}$] affine spaces over the truss $\star$ with isotropic $\star$-groups actions; \pageref{p-af-star-is}

\item[$\Aff_T$] $T$-affine spaces; \pageref{p-aff-t}

\item[$\Ah$] abelian heaps; \pageref{p-cat-ah}

\item[$\grp$] groups; \pageref{p-grp}

\item[$\heap$] heaps; \pageref{p-cat-heaps}

\item[$\quandle$] quandles; \pageref{p-quandle}

\item[$\Set$] sets;

\item[$\spin$] spindles; \pageref{p-quandle}

\item[$\MAbs$] $T$-modules with a fixed absorber; \pageref{p-t-abs}

\item[$T\GMod$] $T$-groups; \pageref{p-t-gmod}

\item[$T\GMod_{\mathbf{is}}$] isotropic $T$-groups; \pageref{p-t-gmod}

\item[$T\HMod$] heaps of $T$-modules; \pageref{p-t-hmod} 

\item[$T\HMod^{\mathbf{cn}}$] contractible heaps of $T$-modules; \pageref{p-contr-mod}

\item[$T\HMod^{\mathbf{cn}}_{\mathbf{is}}$] contractible isotropic heaps of $T$-modules; \pageref{p-contr-is-heaps}

\item[$T\HMod_{\mathbf{is}}$] isotropic heaps of $T$-modules; \pageref{p-isotropic-h}

\item[$T\Mod$] $T$-modules; \pageref{p-t-mod}

\item[$T\Mod_{\mathbf{is}}$] isotropic $T$-modules; \pageref{page-iso-t-mod}

\item[$\left(\tT(\FF)\HMod_{\mathbf{is}}^{\mathbf{cn}}\right)^{\mathbf{in}}$] inhabited isotropic and contractible heaps of $\tT(\FF)$-modules; \pageref{cor:afffield}

\end{description}

\smallskip

\subsection{Functors}

\begin{description}[itemsep=1.5pt,leftmargin=19pt]

\item[$\hH$, $\hH(G)$] functor assigning a heap to a group, heap associated with $G$; \pageref{functor-heap}

\item[$\mathcal{H}$] functor assigning a heap of $T$-modules to a $T$-module via induced actions; \pageref{functor-Hh}

\item[$\tT$, $\tT(R)$] functor assigning a truss to a ring,  truss associated with $R$; \pageref{p-ass-truss}

\end{description}

\smallskip

\subsection{Special objects}

\begin{description}[itemsep=1.5pt,leftmargin=19pt]

\item[$\star$] singleton heap, singleton truss, singleton module, singleton heap of modules; \pageref{page-star}

\item[$\mathrm{Ann}(M)$] annihilator or contracting paragon of a heap of $T$-modules; Lemma~\ref{lem:annihilator}

\item[$\mathrm{Ann}_e(M)$] $e$-annihilator or $e$-contracting paragon of a $T$-module; Lemma~\ref{lem:annMod}

\item[$\eE(H)$] endomorphism truss of an abelian heap $H$; \pageref{page-end-tr}

\item[$\eE_T(M)$] endomorphism truss of an heap of $T$-modules $M$; \pageref{page:ETM}

\item[$\gG(H;e)$] retract of a heap $H$ at $e\in H$; \pageref{groups-retract}

\item[$\rR(T)$] ring associated with a truss $T$ by extension $T\boxplus \star$ through the singleton truss; \pageref{rR(M)}

\item[$\rR(M)$] module over the ring $\rR(T)$ associated with the $T$-module $M$; \pageref{rR(M)}

\item[$\mathrm{Stab}(M)$] stabiliser or isotropy paragon of a $T$-module $M$ (see Lemma~\ref{lem:stab}), stabiliser or isotropy paragon of a heap of $T$-modules $M$ (see Lemma~\ref{lem:isotropicheap});

\item[$\mathrm{Trans}(H)$] translation group of $H$; \pageref{p-trans-group}

\end{description}

\smallskip

\subsection{Other notations}

\begin{description}[itemsep=1.5pt,leftmargin=19pt]

\item[$\sim_S$] sub-heap equivalence relation; \pageref{page-sub-heap-congr}

\item[$\tau_a^b$] translation automorphism; \pageref{page-trans-auto}

\item[$\triangleright_e$] $e$-induced action; \pageref{def:induced}

\item[$\big(M,\act_e\big)$] associated $T$-module; Lemma \ref{lem:homtoind}

\item[$\Sym{X}$] the symmetric group on a set $X$; \pageref{def:Taffinespace}

\end{description}

\end{multicols}

}

\end{appendix}

\end{document}